\documentclass[a4paper, 12pt, one column]{article}

\usepackage[english]{babel}
\usepackage[utf8x]{inputenc}
\usepackage[T1]{fontenc}
\usepackage{comment}

\usepackage[top=1.3cm, bottom=2.0cm, outer=2.5cm, inner=2.5cm, heightrounded,
marginparwidth=1.5cm, marginparsep=0.4cm, margin=2.5cm]{geometry}

\usepackage{graphicx} 
\usepackage[colorlinks=False]{hyperref} 
\usepackage{amsmath}  
\usepackage{amsfonts} %
\usepackage{amssymb}  %
\usepackage{amsthm}
\usepackage{enumerate}
\usepackage{xcolor}

\makeatletter
\newtheorem*{rep@theorem}{\rep@title}
\newcommand{\newreptheorem}[2]{%
\newenvironment{rep#1}[1]{%
 \def\rep@title{#2 \ref{##1}}%
 \begin{rep@theorem}}%
 {\end{rep@theorem}}}
\makeatother

\usepackage[colorinlistoftodos]{todonotes}

\newtheoremstyle{mystyle}{}{}{\slshape}{2pt}{\scshape}{.}{ }{} 

\newtheorem{theorem}{Theorem}[section]
\newreptheorem{theorem}{Theorem}
\newtheorem{cor}[theorem]{Corollary}

\newtheorem{prop}[theorem]{Proposition}
\newtheorem{lemma}[theorem]{Lemma}

\newtheorem{fact}[theorem]{Fact}

\theoremstyle{definition}
\newtheorem{defi}[theorem]{Definition}
\newtheorem{ex}[theorem]{Example}
\theoremstyle{mystyle}
\theoremstyle{remark}
\newtheorem{rem}[theorem]{Remark}
\newtheorem*{claim}{Claim}

\newenvironment{proof of claim}
{\begin{trivlist}  \item \textit{Proof of Claim.}~} {\hfill $\Box$ (Claim)
\end{trivlist}}

\newcommand{\monster}{\mathcal U}
\newcommand{\res}{\mathrm {res}}

\newcommand{\RV}{\mathrm {RV}}
\newcommand{\rv}{\mathrm {rv}}
\newcommand{\trdeg}{\mathrm {trdeg}}
\newcommand{\Aut}{\mathrm {Aut}}
\newcommand{\id}{\mathrm {Id}}

\newcommand{\calL}{\mathcal L}
\newcommand{\Ltilde}{\widetilde{\mathcal L}}
\newcommand{\calS}{\mathcal S}
\newcommand{\Utilde}{\widetilde{\mathcal U}}

\newcommand{\Calg}{C^{\mathrm {alg}}}
\newcommand{\alg}{\mathrm {alg}}
\newcommand{\St}{\mathrm {St}}
\newcommand{\bQ}{\mathbb {Q}}

\DeclareMathOperator{\tp}{tp}

\DeclareMathOperator{\acl}{acl}

\DeclareMathOperator{\dcl}{dcl}

\makeatletter

\def\indsym#1#2{%
 \setbox0=\hbox{$\m@th#1x$}%
 \kern\wd0%
 \hbox to 0pt{\hss$\m@th#1\mid$\hbox to 0pt{$\m@th#1^{#2}$\hss}\hss}%
 \lower.9\ht0\hbox to 0pt{\hss$\m@th#1\smile$\hss}%
 \kern\wd0}
\newcommand{\ind}[1][]{\mathop{\mathpalette\indsym{#1}}}

\def\nindsym#1#2{%
 \setbox0=\hbox{$\m@th#1x$}%
 \kern\wd0%
 \hbox to 0pt{\hss$\m@th#1\not$\kern1.4\wd0\hss}
 \hbox to 0pt{\hss$\m@th#1\mid$\hbox to 0pt{$\m@th#1^{#2}$\hss}\hss}%
 \lower.9\ht0\hbox to 0pt{\hss$\m@th#1\smile$\hss}%
 \kern\wd0}

\newcommand{\aclind}{\ind^{\alg}}

\title{Residue field domination in some henselian valued fields}
\author{C. Ealy, D. Haskell,  and P. Simon}

\begin{document}
\maketitle

\begin{abstract}
We generalize previous results about stable domination and residue field domination to henselian valued fields of equicharacteristic 0 with bounded Galois group, and we provide an alternate characterization of stable domination in algebraically closed valued fields for types over parameters in the field sort.
\end{abstract}

\section{Introduction}
 
The notion of domination of a type by its stable part was introduced and studied in the book \cite{HHM2} and examined especially in the case of an algebraically closed valued field. The utility of the notion has been further demonstrated; for example, the space of stably dominated types in an algebraically closed valued field was analysed in the book \cite{HL} as an approach to understanding Berkovich spaces, and some structure theory has been developed for groups with a stably dominated generic type \cite{HR-K}. However, the stable part of a structure can seem like an unwieldy and abstract object. Since the stable sorts in an algebraically closed valued field are essentially those which are internal to the residue field, the intuition behind stable domination is that a stably dominated type is controlled by its trace in the residue field. By turning attention to the residue field instead of to the stable part, the hope is that this intuition could be used in two ways. The first is to develop a notion of domination that applies in more general valued fields in which the residue field is not necessarily stable. The second is to find a domination statement involving a simpler collection of sorts. This program was started in \cite{EHM}, where we considered domination by sorts that are internal to the residue field in a  real closed valued field. The present paper continues the project in the greater generality of henselian valued fields of equicharacteristic $0$, provided that the galois group is bounded. Details of the notation are given later; in the theorems quoted below, $\monster$ is a monster model of the theory of valued fields in which we are working.

In our definition of residue field domination, we reduce the collection of sorts that are used for domination to the residue field itself, rather than the sorts that are internal to the residue field. This may seem to be an unreasonably strong property, but we are able to show that it does hold in many cases, either assuming some algebraic conditions, or assuming stable domination, as in the following statements.

\begin{reptheorem}{residue field domination for unramified extensions}
  Let $C\subseteq \monster$ be a subfield and let $a$ be a (possibly infinite) tuple of field elements such that the field generated by $Ca$ is an unramified extension of $C$ with the good separated basis property over $C$, and such that $k(Ca)$ is a regular extension of $k(C)$.  Then $\tp(a/C)$ is residue field dominated.
\end{reptheorem}

\begin{reptheorem}{theorem:domfield}
Let $C\subseteq \monster$ be a subfield, let $a\in \monster$, and let $\widetilde \monster$ be the algebraic closure of $\monster$. Assume that $\tp(a/C)$ is stably dominated in the structure $\widetilde \monster$.  Then in the structure $\monster$, $\tp(a/C^+)$ is residue field dominated, where $C^+ = \acl(C)\cap \dcl(Ca)$.
\end{reptheorem}

There are, however, important examples, when the base of a type is not in the field sort, where stable domination does not reduce to residue field domination.  For instance, a major theme of stable domination is that types (with a few caveats) are always stably dominated over the value group. However, they need not be residue field dominated over the value group. In addition to the residue field, one needs information from sorts that are internal to the residue field. These turn out to be given by fibers of the valuation map in $\RV$.  We thus introduce another notion, $\RV$-domination, and show that types are $\RV$-dominated over their value groups.

\begin{reptheorem}{RVdomination}
Let $L$, $M$ be subfields of $\mathcal{U}$ with $C\subseteq L\cap M$ a valued subfield. Assume that $k(L)$ is a regular extension of $k(C)$, $\Gamma_L \subseteq \Gamma_M$, $\Gamma_L/\Gamma_C$ is torsion-free and that $L$ has the good separated basis property over $C$.  Then $\tp(L/C\Gamma_L)$ is $\RV$-dominated.
\end{reptheorem}
 
An important insight of this paper is that one key step in proving domination results is the existence of a separated basis. This insight allows us to distinguish between purely algebraic concepts and the more model-theoretic ones. In particular, we derive the following algebraic characterization of stable domination for types in the field sort in an algebraically closed valued field.

\begin{reptheorem}{domination-equivalence}
Suppose that $\monster$ is algebraically closed. Let $C \subset \monster$ be a subfield, let $a$ be a tuple of valued field elements, and let $L$ be the definable closure of $Ca$ in the valued field sort.  Assume $L$ is a regular extension of $C$. Then the following are equivalent.
\begin{enumerate}[(i)]
    \item $\tp(a/C)$ is stably dominated.
    \item $L$ has the good separated basis property over $C$ and $L$ is an unramified extension of $C$.
\end{enumerate} 
\end{reptheorem}

When restricted to the main sort, the domination statements can be given a purely valuation-theoretic form, as asserting the existence of automorphisms under certain hypotheses; these are  Proposition~\ref{basicisomorphismtheorem} and Theorem~\ref{dominationovervaluegroup}.

\medskip
In the time since this paper was originally submitted, further work has been done by several authors. We mention in particular the work of Vicaria \cite{V}, which uses, and to some extent generalizes, the results of this paper.  She does not need the hypothesis that the Galois group of the field is bounded. However, she uses a rather different language, with sorts for the cosets of the subgroups of the $n$th powers in $\RV$. Also relevant is the work of Hils, Cubides Kovacsics and Ye \cite{HCY}, which independently obtains type implication results using the existence of a separated basis (there called being vs-defectless).
 
\medskip
The outline of the paper is as follows. In the remainder of the Introduction we state a quantifier elimination result for the theory in which we work,  give the definition of domination and some associated properties, and recall some elementary properties of type implication and regular field extensions. In Section 2, we define the notion of a good separated basis over a base field $C$ and some consequences, in particular the relation to the assumption that $C$ is a maximal field. In Section 3, we prove some preliminary results towards residue field domination, using the separated basis hypothesis. Finally, in Section 4 we derive the full domination results, after showing that the geometric sorts can be resolved in the field sort.

\subsection{Notation}

We work in two languages, $\calL$ and $\Ltilde$, and two structures, $\monster$ and $\widetilde \monster$. 

We fix $K$, a henselian valued field of equicharacteristic $0$ with bounded Galois group.  The first language, $\calL$, is described in Proposition \ref{qe} below; it depends on $K$. We fix the theory $T$ of $K$ in the language $\calL$.  We let $\monster$ be a monster model of $T$.

The second language, $\Ltilde$, is the language often used for algebraically closed valued fields. We equip the field sort with the usual ring language and use the notation $k$ for the residue field sort in the usual ring language, $\Gamma$ for the value group sort in the language of ordered abelian groups and $\RV$ for the RV-sort with the induced multiplicative group structure. We include the geometric sorts required to eliminate imaginaries, namely $\bigcup_{n=1}^\infty S_n$ for the lattices and $\bigcup_{n=1}^\infty T_n $ for their torsors.  However, the resolution results of  Theorem~\ref{resolution} below and Chapter 11 of \cite{HHM2} allow us to avoid working with the geometric sorts directly in this paper, and thus we will omit their (rather lengthy) definition; a detailed description can be found in \cite[Section 3.1]{HHM} and \cite[Section 7.4]{HHM2}.

We let $\widetilde \monster$ be a monster model of ACVF such that the field sort of $\monster$  embeds into the field sort of $\widetilde\monster$, and such that every automorphism of $\monster$ extends to an automorphism of $\widetilde \monster$ (e.g. $\widetilde \monster$ could be the algebraic closure of $\monster$). 
Throughout the paper, we will use a subscript $\Ltilde$ to indicate not just that we are working in the language $\Ltilde$, but that we are also working in the algebraically closed valued field $\widetilde \monster$ (for instance, when taking definable closure, or specifying a type); no subscript indicates that we are working in the language $\calL$ and in $\monster$. 

Given any definable set $\calS$ and any set of parameters $C$, we write $\calS(C) = \dcl(C)\cap\calS$. If $C$ is a substructure of $\monster$, we write $\calS_C=C\cap \calS$. 
 For any field, we use the superscript $\alg$ to denote its field theoretic algebraic closure. On any field, and in particular on the residue field $k$, we have an independence relation $\aclind$:  for $A,B \subseteq k$, $A\aclind_C B$ means that any finite subset of $k(AC)$ that is field algebraically independent over $k(C)$ remains so over $k(BC)$.

\subsection{Quantifier elimination}

The language $\calL$ is chosen so that the theory of the valued field that we are working with has quantifier elimination. This is derived from the following results as described below. The first is a result of Chernikov and Simon translated into the notation of valued fields.  Note that bounded Galois group implies that the $n$th powers have finite index in the field \cite{FehmJahnke} and hence also in $\RV$.  This is our paper's only use of the assumption of bounded Galois group.  One may construct henselian fields of equicharacteristic 0 where $n$th powers have finite index in $\RV$  but which do not have bounded Galois group (see \cite[Proposition 5.1]{FehmJahnke}).  Our results apply to these fields as well.

\begin{fact}\cite[Proposition 3.1]{hens_inp}\label{QE-for-RV}
Let $K$ be a henselian valued field of equicharacteristic $0$ with bounded Galois group. 
Assume the language $\calL$ is chosen so that:
\begin{itemize}
	\item  $\RV$ has its multiplicative group structure, a predicate for $k$ as a multiplicative subgroup, $n$th power predicates, constants naming a countable subgroup containing representatives of the (finitely many) cosets of the $n$th powers for $n<\omega$ (where representatives of classes which intersect $k$ are chosen in $k$), a sort for $\Gamma$, and a map $v:\RV\to\Gamma$; 
	\item  the language of $\Gamma$ expands the structure induced from $K$,  has no function symbols apart from $+$, and eliminates quantifiers;
	\item  the language of $k$ expands the structure induced from $K$,  has no function symbols apart from $\cdot$, and eliminates quantifiers.
\end{itemize}

Then $(\RV,\Gamma, k)$ has quantifier elimination.
\end{fact}

\begin{fact} {\cite[Theorem 4.1]{Pas}}
Let T be the theory of a henselian valued field of equicharacteristic 0, in the language with sorts for $k$ and $\Gamma$, expanded
by the angular component map. Then T has elimination of field quantifiers.
\end{fact}

One can show (e.g. \cite{Cluckers}, \cite[Theorem A]{RK} or  \cite[Corollary 5.8,  assuming the trivial derivation]{Scanlon}) that elimination of field quantifiers with an angular component map implies elimination of field quantifiers relative to $\RV$.  In our case, $\RV$ itself eliminates quantifiers as in Fact \ref{QE-for-RV}, and thus we may conclude Proposition~\ref{qe} below.  We remark that the form in which this proposition is generally used is the following: if $A,B\subset \monster$ are valued fields, and $\sigma:A \to B$ is a valued field isomorphism which induces an isomorphism of $\RV$-structures $\RV_A \to \RV_B$, then $\sigma$ extends to an automorphism of $\monster$.
\begin{prop}\label{qe}

Let $K$ be a henselian valued field of equicharacteristic $0$ with bounded Galois group.  Work in the language with 
\begin{itemize}
    \item the language of rings on $K$,
    \item a sort for $\RV$ and a sort for $\Gamma$, each in the language of groups, 
    \item a predicate for $k\subset \RV$, 
    \item a map $\rv:K\to \RV$, 
    \item a map $v:\RV \to \Gamma$, 
    \item predicates for every subset of $k^m$ and $\Gamma^m$ definable without parameters in the structure induced from $K$, 
    \item predicates for the nth powers in $\RV$, and 
    \item constants for a countable subgroup of $\RV$ containing coset representatives for each of the nth power subgroups of $\RV$, chosen in $k$ where possible.

\end{itemize}
Then $K$ has quantifier elimination.

\end{prop}

\begin{rem}\label{stably embedded}
It follows from this proposition that the value group and residue field  are stably embedded in the following strong form: if $\varphi(x,a)$ defines a subset of $k^n$, then there is a term $t$ and quantifier free formula $\theta$ such that  $\theta(x,t(a))$ defines the same subset.  Given that $\theta$ is quantifier free, it is clear that $t(a)$ lies in the $\RV$-structure (either in $\RV$ itself or in $\Gamma$).  It is easy to check that $t(a)$ can be chosen to lie in the residue field.  The same argument also shows that if $X$ is a subset of $\Gamma$ defined over $a$ then it is also defined over $t(a)\in \Gamma$ for some term $t$.
Note that this is slightly stronger than the definition of stable embeddedness, which does not require the parameter in the stably embedded set to be in $\dcl(a)$. 

 We would not in general expect this strong form of stable embeddedness to hold for an individual fibre in $\RV$ which we write as $\RV_\gamma=\{x\in \RV : v(x)=\gamma\}$. For consider the subset of $\RV_\gamma \times \RV_\gamma$ defined by $x\cdot y^{-1}=a$ where $a\in k$.  However, if one assumes that $\RV_\gamma$ contains some point $a_0$ that is expressible as a term $t_0(a)$, then it is again true that any definable subset of ${\RV_\gamma}^n$ defined over $a$ is defined over a term $t(a)$ with $t(a)\in \RV_\gamma$.  For if $X$ is such a set, $X\cdot {a_0}^{-1}$ is a definable subset of the residue field, and therefore definable over $t'(a)\in k$ for some term $t'$.  Hence $X\cdot {a_0}^{-1}$ is also definable over $t'(a)\cdot t_0(a)\in \RV_\gamma$, and so is $X$.

Lastly, the quantifier elimination result implies that the residue field and value group are orthogonal to each other.
\end{rem}

\subsection{Domination: definition and basic properties}

Residue field domination is defined by analogy with stable domination, which we now recall (\cite[Definition 3.9]{HHM2}). Given a set of parameters $C$ in $\widetilde{\monster}$, let $\St_C$ be the multi-sorted structure whose sorts are the $C$-definable stable, stably embedded subsets of $\widetilde{\monster}$. The structure $\St_C$ is itself stable, so stable forking gives an independence relation $\ind$.

\begin{defi}\label{stable domination definition}
We say that $\tp_{\Ltilde}(a/C)$ is {\em stably dominated} if for any $b\in\widetilde{\monster}$, whenever $\St_C(aC) \ind_C \St_C(bC)$ we have $\tp_{\Ltilde}(b/C\St_C(aC)) \vdash \tp_{\Ltilde}(b/Ca)$.
\end{defi}

The definition captures our intuition that a stably dominated type should have no interaction with the value group in the following sense.

\begin{fact}\label{stably dominated implies orthogonal}\cite[Corollary 10.8]{HHM2} 
 $\tp_{\Ltilde}(a/C)$ is stably dominated if and only if it is orthogonal to $\Gamma$. 
\end{fact}
Notice that Corollary 10.8 and the definition of orthogonality in \cite[Definition 10.1]{HHM2} are only given in the original for the case when $a$ is a unary sequence. However they both can be stated in more generality, since for any element $s$ and any set $C$ in the geometric sorts of a valued field, there is a unary sequence, $a$, with the same $\Ltilde$-definable closure over $C$ (\cite[Proposition 2.3.10]{HHM} or \cite[Proposition 7.14]{HHM2}). For such an $s$ and $a$, one may define $\tp(s/C)$ to be orthogonal to $\Gamma$ if $\tp(a/C)$ is orthogonal to $\Gamma$, noting by \cite[Lemma 10.9]{HHM2} that this is independent of the choice of $a$.

The structure $\St_C$ can be defined in any structure, but it may be trivial or hard to identify. In an algebraically closed valued field, $\St_C$ is interdefinable with the collection of sorts internal to the residue field, which are themselves interdefinable (with parameters) with the residue field. This motivates the following definition for a valued field that is not necessarily algebraically closed. Notice that residue field domination as defined here is a very strong property, since the independence notion we are working with is very weak.  It is thus surprising that we can prove instances of residue field domination in Section~4.

\begin{defi}\label{domination definition}
We say that $\tp(a/C)$ is {\em residue field dominated} if for any $b\in \monster$, if $k(aC)\aclind_C k(bC)$, then $\tp(b/C k(Ca))\vdash \tp(b/Ca)$.  
\end{defi}

\noindent When $\monster$ is itself algebraically closed, it is immediate that residue field domination implies stable domination. If $\monster$ is, for example, a real closed valued field, this implication does not hold.  The converse is not true even when $\monster$ is algebraically closed, as the following example illustrates. In particular, this example shows that issues may arise when the type is over parameters in the value group sort. 

\begin{ex}\label{stably not residue field dominated}
Let $C=\bQ$ and let $a\in\monster$ be a field element of positive valuation. Then $C$ is maximal because it is trivially valued, $L=\dcl(a)$ has $k_L = k_C$ and hence is automatically a regular extension, and $\Gamma_L$ is a torsion free extension of $\Gamma_C$ (which is the trivial group). So by \cite[Theorem 12.18]{HHM}, $\tp(a/C\Gamma_L)$ is stably dominated. However, $\tp(a/C\Gamma_L)$ is not residue field dominated. For if we take $M=L$,  the independence condition holds trivially since $k_M=k_L=k_C$, but it is not the case that $\tp(L/C\Gamma_Lk_L)$ implies $\tp(L/M)=\tp(L/L)$.
\end{ex}

\noindent We are able to prove a version of \cite[Theorem 12.18]{HHM}, involving $RV$-domination instead of residue field domination, which we will define in Section~3.1.

 In \cite{HHM}, it is shown that stable domination is insensitive to whether or not the base is algebraically closed.

\begin{fact}\label{stable domination lifts to acl}\cite[Corollary 3.31]{HHM}
$\tp(a/C)$ is stably dominated if and only if $\tp(a/\acl(C))$ is stably dominated. 
\end{fact}

\noindent This is not true for residue field domination, as the following example illustrates. We make use here, and many times later, of the following basic fact.

\begin{fact}
Let $C\subset\Utilde$, $a\in\Utilde$. Then $\dcl_{\Ltilde}(Ca)$ (restricted to the field sort) is the henselization of the field generated by $a$
 over $C$.
 \end{fact}

\begin{ex}
Let $K$ be an algebraically closed valued field of characteristic 0, let $t$ be an element of positive valuation, and consider $C=\dcl(\mathbb{Q}(t))$.  We note that $\sqrt{t}$ cannot be in $C$ since the definable closure of $\mathbb{Q}(t)$ is the henselization of $\mathbb{Q}(t)$, which is an immediate extension.   Let $a=\sqrt{t}$.  Clearly $\tp(a/\acl(C))$ is stably dominated and residue field dominated.  Yet $\tp(a/C)$ is stably dominated but not residue field dominated.  To see the second statement, choose $b=a$.  One has $k(aC)\aclind_C k(aC)$ since $a\in \acl(C)$.  Since $\sqrt{t}$ generates a ramified extension of $C$, $k(Ca)=k(C)$.  Thus $\tp(a/Ck(Ca))=\tp(a/C)$, and clearly $\tp(a/C)$ cannot imply $\tp(a/Ca)$.

On the other hand, $\tp(a/C)$ is stably dominated.  Since $a\in\acl(C)$, $a$ is in a $C$-definable stable, stably embbeded set, i.e. is in $\St(C)$. So automatically $\tp(b/C\St_C(a))$ implies $\tp(b/Ca)$ for any $b$.
\end{ex}

\noindent However we do get the following, slightly weaker, statement.  The proof uses Proposition~\ref{resdominationequivalent} below. 

 \begin{prop}\label{base acl}
For $C\subset \monster$ and $a\in\monster$, let $C^+ = \acl(C)\cap \dcl(Ca)$. Then $\tp(a/C^+)$ is residue field dominated if and only if $\tp(a/\acl(C))$ is residue field dominated.
\end{prop}

\begin{proof}
For the right-to-left direction, choose $b$ such that $k(C^+a)\aclind_{C^+} k(C^+b)$.  Since fields code finite sets, if $d_1\in \acl(C)$ and the orbit of $d_1$ over $C$ is $d_1, \dots, d_n$, then $\{d_1,\dots, d_n\}\in\dcl(C)$ and $d_1, \dots, d_n \in \textrm{alg}(\dcl(C))$, where $\textrm{alg}$ denotes the field theoretic algebraic closure.  Thus $\acl(C)\subseteq \textrm{alg}(C^+)$. Note that $k(C^+a)\aclind_{C^+} k(C^+b)$ implies $\textrm{alg}(k(C^+a))\aclind_{\textrm{alg}(C^+)} \textrm{alg}(k(C^+b))$, which in turn implies $k(\textrm{alg}(C^+a))\aclind_{\textrm{alg}(C^+)} k(\textrm{alg}(C^+b))$, which implies 
 $k(\acl(C)a)\aclind_{\acl(C)} k(\acl(C)b)$.  
 
 Thus one has $\tp(b/\acl(C)k(\acl(C)a))\vdash\tp(b/\acl(C)a)$ and one wants $\tp(b/C)\vdash\tp(b/C^+a)$.   Choose $\varphi(x,a)\in \tp(b/C^+a)$.  This is implied by some $\psi(x,c, d)\in \tp(b/\acl(C)k(\acl(C)a))$, with $c\in \acl(C)$ and $d\in k(\acl(C)a)$.  
Let $X = \{ \sigma(c)\sigma(d) : \sigma \in \Aut(\monster/C^+a)\}$. Notice that $X_1 = \{ \sigma(c) : \sigma \in \Aut(\monster/C^+a)\}$ is finite, so $Ca$-definable, and in $\acl(C)$, hence fixed by any automorphism fixing $C^+$. Also $X_2 = \{ \sigma(d) : \sigma \in \Aut(\monster/C^+a)\}$ is $C^+a$-definable and in the residue field, and thus $X_2\in k(C^+a)$.

Thus the formula, $\theta_0$, given by
\[ 
	\bigvee_{\sigma(c)\in X_1} \bigvee_{\sigma'(d)\in X_2} \psi(x,\sigma(c),\sigma'(d))
\]
is over $C^+k(C^+a)$ as desired, and for any $\sigma(c)\sigma(d)$ in $X$, we know that $\psi(x,\sigma(c),\sigma(d))$ implies $\varphi(x,a)$. However, if $\sigma'$ is some other automorphism fixing $C^+a$, it may be the case that $\psi(x,\sigma(c),\sigma'(d))$ does not imply $\varphi(x,a)$, and so we must tweak $\theta_0$. If $\sigma'$ is such an isomorphism, then $\sigma(c)\sigma'(d)\not\equiv_{C^+a} \sigma(c)\sigma(d)$ and thus 
$\sigma'(d) \not\equiv_{\sigma(c) k(C^+a)} \sigma(d)$. For each $\sigma\in\Aut(\monster/C^+a)$, let $e_{\sigma(c)}$ be the orbit of $\sigma(d)$ over $\sigma(c)k(C^+a)$. Then the formula, $\theta$, given by 
\[ 
	\bigvee_{\sigma(c)\in X_1} \bigvee_{d'\in e_{\sigma(c)}} \psi(x,\sigma(c),d')
\]
 implies $\varphi(x,a)$.

 We claim that $\{\sigma(c)e_{\sigma}(c) : \sigma\in \Aut(\monster/C^+a)\}$ is $C^+k(C^+a)$-definable, and hence the displayed formula above gives the required domination statement. Consider $\tau$ an automorphism fixing $C^+k(C^+a)$. Since $\tau$ fixes $C^+$, $\tau$ maps $X_1$ to itself, so there is an automorphism $\sigma$ fixing $C^+a$ such that $\tau(c) = \sigma(c)$. It suffices to show that $\tau(d)\in e_{\sigma(c)}$. By definition, $\sigma(d)\in e_{\sigma(c)}$. Now $\tau\circ\sigma^{-1}$ fixes $\sigma(c)$ and $k(C^+a)$, and $\tau\circ\sigma^{-1}(\sigma(d)) = \tau(d)$, which hence lies in the $\Aut(\monster/\sigma(c)k(C^+a))$-orbit of $\sigma(d$), as required.

For the other direction, take $b$ with $k(\acl(C)a)\aclind_{\acl(C)} k(\acl(C)b)$.  It suffices, by Proposition \ref{resdominationequivalent}, to show that $\tp(a/\acl(C)k(\acl(C)b))\vdash\tp(a/\acl(C)b)$.  Note that, by replacing the set $k(\acl(C)a)$ with a subset and replacing the set $\acl(C)$ in the base with something interalgebraic with it, we have 
 \[ k(C^+a)\aclind_{C^+} k(\acl(C)b).
 \]
Thus we may apply residue field domination of $\tp(a/C^+)$ where our tuple from $\monster$ is $\acl(C)b$, obtaining (again applying Proposition \ref{resdominationequivalent}) $\tp(a/C^+k(\acl(C)b))\vdash\tp(a/\acl(C)b)$.  So certainly \[ \tp(a/\acl(C)k(\acl(C)b))\vdash\tp(a/\acl(C)b)\] 
as well.

 \end{proof}

\subsection{Type implications}

Since many of our arguments involve showing type implications, it is useful to make the following very general observations.

\begin{lemma}\label{type implications}
Let $A$, $B$, $C$ be subsets of a monster model $\monster$ in some language, with $C\subseteq A\cap B$. Then
\begin{enumerate}[(i)]
    \item $\tp(A/C) \vdash \tp(A/B)$ is equivalent to $\tp(B/C) \vdash \tp(B/A)$;
    \item if $\tp(A/C) \vdash \tp(A/B)$  and $\tp(B'/C) = \tp(B/C)$ then $\tp(A/C) \vdash \tp(A/B')$ .
\end{enumerate}
\end{lemma}

\begin{proof}
\begin{enumerate}[(i)]
    \item Suppose $\tp(A/C) \vdash \tp(A/B)$ and $\tp(B'/C) = \tp(B/C)$. Let $\sigma\in\Aut(\monster/C)$ with $\sigma(B')=B$. As $\tp(\sigma(A)/C) = \tp(A/C)$, by the type implication assumption, also $\tp(\sigma(A)/B) = \tp(A/B)$. Thus there is $\tau\in\Aut(\monster/B)$ such that $\tau(\sigma(A))=A$. Then $\tau(\sigma(B'))=B$, so $\tp(B'/A) = \tp(B/A)$.

    \item By (i), it is equivalent to show that $\tp(B'/C)\vdash \tp(B'/A)$, which is the same statement as $\tp(B/C)\vdash \tp(B'/A)$ .  Also by (i), we have $\tp(B/C) \vdash \tp(B/A)$.  So we need only establish that $\tp(B'/A)=\tp(B/A)$.  But since we know that $\tp(B/C) \vdash \tp(B/A)$, we know that anything (e.g. $B'$) that realizes $\tp(B/C)$ must also realize $\tp(B/A)$. Thus $B'\models \tp(B/A)$ and $\tp(B'/A)=\tp(B/A)$.
    
\end{enumerate}
\end{proof}

The following lemma is stated in \cite[Remark 3.7]{HHM2} for the stable part of a structure. We prove it here using Remark \ref{stably embedded} which allows us to avoid the assumption of elimination of imaginaries. Let $\calS$ be any definable set that is stably embedded in the strong sense defined in Remark~\ref{stably embedded}. Later we will take $\calS$ to be the residue field, the value group, or some collection of fibres of $\RV$ where for each $\gamma$, $\RV_\gamma(CB)$ is non-empty.

\begin{lemma}\label{Remark 3.7 but no e of i}
For any sets $A$, $B$, $C$ in $\monster$, $\tp(B/C\calS(CB)) \vdash \tp(B/C\calS(CB)\calS(CA))$.
\end{lemma}

\begin{proof}
We may assume $B$ is finite. Take $B' \equiv_{C\calS(CB)} B$.  We wish to show that $B' \equiv_{C\calS(CB)\calS(CA)} B$, so take $\varphi(x,a,b) \in \tp(B/C\calS(CA)\calS(CB))$ with $a\in \calS(CA)$ and $b\in \calS(CB)$.  We wish to show that $\varphi(B',a,b)$ holds.

Consider the set defined by $\varphi(B,y, b)$.  This is a subset of $\calS$, defined over $CB$, and hence definable by some $\theta(y,\tilde b)$ where $\tilde b \in \calS(CB)$ as described in Remark \ref{stably embedded}. Thus $\forall y \; [\theta(y,\tilde b) \rightarrow \varphi(x,y,b)] \in \tp(B/C\calS(CB))$.

Since $\forall y \; [\theta(y,\tilde b) \rightarrow \varphi(B',y,b)]$ holds and $\theta(a,\tilde b)$ also holds, it follows that $\varphi(B',a,b)$ holds. 
\end{proof}

From this, we derive equivalences for the type implication in the definition of residue field domination.

\begin{prop}\label{resdominationequivalent}
For any $a, b, C$ in $\monster$ the following are equivalent:
\begin{enumerate}[(i)]
    \item $\tp(b/C \calS(Ca))\vdash \tp(b/Ca)$;
    \item $\tp(a/C\calS(Cb))\vdash \tp(a/Cb)$;
    \item $\tp(\calS(bC)/C\calS(Ca)) \cup \tp(b/C) \vdash \tp(b/Ca)$;
    \item $\tp(a/C\calS(Ca)\calS(Cb))\vdash \tp(a/Cb)$.
    \item $\tp(a/C\calS(Ca))\vdash \tp(a/Cb)$
\end{enumerate}
\end{prop}

\begin{proof}
The proof of the equivalence of (i), (ii) and (iii) is exactly the proof of \cite[Lemma 3.8]{HHM2}, replacing the stable, stably embedded sorts with the definable set $\calS$, and referring to Lemma \ref{Remark 3.7 but no e of i} in lieu of \cite[Remark 3.7]{HHM2}.  The fact that (ii) implies (iv) is trivial, and that (iv) implies (v) is immediate by Lemma \ref{Remark 3.7 but no e of i}.

To prove that (v) implies (i), assume (v).  Take $b, b'\models \tp(b/C\calS(Ca))$ and $\sigma$ witnessing this.   Suppose that $\sigma^{-1}(a)=\tilde a$ and note that $a, \tilde a \models \tp(a/C\calS(Ca))$ and thus by $(v)$ they both satisfy $\tp(a/Cb)$.  Choose $\tau:a\mapsto \tilde a$ witnessing this.  Thus $(\sigma\circ\tau)(a)=\sigma(\tilde a)=a$ and $(\sigma \circ \tau)(b)=\sigma(b)=b'$, and we see that $b, b'\models \tp(b/Ca)$.  
\end{proof}

\noindent We will have need of the following result, which we will use in the form of the subsequent lemma.

\begin{fact}\label{HHM2 equality of value groups}\cite[8.22(ii)]{HHM2}
 Let $C\subseteq A,B$ be algebraically closed valued fields and suppose that $\Gamma(C) = \Gamma(A)$, the transcendence degree of $B$ over $C$ is 1, and there is no embedding of $B$ into $A$ over $C$. Then $\Gamma(AB)=\Gamma(B)$.
\end{fact}

\noindent Recall that we use $\Gamma(C)$ to mean $\dcl(C)\cap \Gamma$.  In the following lemma, as we are working in $\widetilde \monster$, the definable closure is taken in $\Ltilde$.

\begin{lemma}\label{Gamma(LF) is Gamma(L)}
    Let $C\subseteq F, L$ be valued fields contained in  $\widetilde{\monster}$ such that $L$ is transcendence degree at least 1 over $C$, $\tp_{\Ltilde}(L/C)\vdash \tp_{\Ltilde}(L/F)$, and $\Gamma(F)=\Gamma(C)$. Then $\Gamma(LF)=\Gamma(L)$. 
\end{lemma}

\begin{proof}
We proceed by induction on the transcendence degree, $n$, of $L$ over $C$.  

Assume $n=1$. Since $\tp_{\Ltilde}(L/C)\vdash \tp_{\Ltilde}(L/F)$, no $\ell\in L\setminus \acl_{\Ltilde}(C)$ can be embedded into $\acl_{\Ltilde}(F)$ over $C$.  For suppose that $\ell \equiv_C \ell'$. Then also $\ell \equiv_F \ell'$.  If $\ell'$ could be chosen in $\acl_{\Ltilde}(F)$, then $\ell$ would be an element of the finite set of elements realizing $\tp_{\Ltilde}(\ell'/F)$.  But this applies equally to any element of $\tp_{\Ltilde}(\ell/C)$, and hence this type has finitely many realizations.  Then $\ell$ would be in $\acl_{\Ltilde}(C)$. Hence there is no embedding of $\acl_{\Ltilde}(L)$ into $\acl_{\Ltilde}(F)$ over $\acl_{\Ltilde}(C)$, and we apply Fact  \ref{HHM2 equality of value groups} to obtain $\Gamma(\acl_{\Ltilde}(L)\acl_{\Ltilde}(F))=\Gamma(\acl_{\Ltilde}(L))$.  Recalling that we have defined $\Gamma(A)$ to be the definable closure of the value group of $A$, we have $\Gamma(LF)=\Gamma(L)$.

Assume the result for $m<n$ and suppose $L$ has transcendence degree $n$ over $C$.  Let $C\subseteq C'\subseteq L$ be such that $L$ is transcendence degree 1 over $C'$. Note that $\tp_{\Ltilde}(L/C')\vdash \tp_{\Ltilde}(L/FC')$ (since $\ell \equiv_{C'} \ell'$ implies $\ell C' \equiv_C \ell' C'$ implies $\ell C' \equiv_F \ell' C'$ implies $\ell \equiv_{C'F} \ell'$). Thus, by our inductive hypothesis, $\Gamma(C'F)=\Gamma(C')$.  Now one may repeat the argument of the $n=1$ case with $C'$ playing the role of $C$ and $C'F$ playing the role of $F$.
\end{proof}

\subsection{Regular extensions}

The following three properties of regular extensions of fields are implicit in many of our arguments. 

\begin{fact}{\cite[VIII, 4.12]{Lang}} \label{Lang}
Suppose $C$ is a field, $L$ is a regular field extension of $C$ and $M$ is any field extension of $C$, all contained in $\widetilde{\monster}$.  Then $L\aclind_C M$ implies $L$ and $M$ are linearly disjoint over $C$. 
\end{fact}

\begin{lemma}\label{orthogonaltoacl}
Let $C$ and $L$ be valued fields contained in $\widetilde{\monster}$ such that  $C\subseteq L$ is a regular extension of fields and $L$ is henselian. Then $\tp_{\Ltilde}(L/C)\vdash \tp_{\Ltilde}(L/\acl_{\Ltilde}(C))$.
\end{lemma}

\begin{proof}
Note we may restrict our attention to the valued field sort of $\widetilde{\monster}$. Let $a\in L$ be a finite tuple. Let $X$ be an $\acl_{\Ltilde}(C)$-definable set containing $a$ and let $X=X_1,\ldots,X_n$ be the conjugates of $X$ over $C$. We may assume that the $X_i$'s are pairwise disjoint (consider the boolean algebra generated by the $X_i$'s and replace $X$ by the atom containing $a$).

Suppose that $X_1$ is defined by $\varphi(x,b)$ with $b\in \acl_{\Ltilde}(C)$.  Consider the set $B$ of conjugates $\{b=b_1, \dots, b_k\}$ of $b$ over $C$, noting that $k$ could be larger than $n$.  Let $S_1$ be the subset of $B$ consisting of those $b_i$ such that $\varphi(x,b_i)$ defines $X_1$.  Since fields code finite sets, there is a tuple $d_1\in \acl_{\Ltilde}(C)$ that is a code for $S_1$.  Consider the conjugates $D=\{d_1, \dots, d_n\}$ of $d_1$ over $C$.  Note that $X_1$ is definable over $d_1$, so it suffices to show that $d_1\in C$.

Since $D$ is $\Ltilde$-definable over $C$, $d_1$ is $\Ltilde$-definable over $Ca$.  Since in an algebraically closed valued field of characteristic 0, the definable closure of a set of field elements is the henselization of the field generated by those elements, $d_1$ is in the henselian closure of $Ca$, which is included in $L$. Since $L$ is a regular extension of $C$ and $d_1$ is algebraic over $C$, we conclude $d_1\in C$ and hence $X$ is ${\Ltilde}$-definable over $C$.

\end{proof}

\begin{lemma}\label{unique extension implies linearly disjoint}
Let $C, F$ and $L$ be valued fields contained in $\widetilde{\monster}$ such that  $C\subseteq F\cap L$, $L$ is a regular extension of $F$, $\tp_{\Ltilde}(L/C)\vdash \tp_{\Ltilde}(L/F)$, and $C$ is not trivially valued. Then $L$ and $F$ are linearly disjoint over $C$.
\end{lemma}
\begin{proof}
By Lemma~\ref{type implications}(i), since $\tp_{\Ltilde}(L/C)\vdash \tp_{\Ltilde}(L/F)$ then also $\tp_{\Ltilde}(F/C)\vdash \tp_{\Ltilde}(F/L)$. Suppose that there is $\vec \ell\in L$ and $\vec f \in F$ such that $\vec \ell\cdot \vec f = 0$ with $\vec \ell, \vec f \neq 0$, and let $\varphi(\vec x, \vec \ell)$ express this of $\vec x$.  As $\varphi(\vec x, \vec \ell)\in \tp_{\Ltilde}(F/L)$ it is implied by some formula $\psi(\vec x,c)\in \tp_{\Ltilde}(F/C)$.  As $\acl_{\Ltilde}(C)$ is a model there is some $\vec d\in \acl_{\Ltilde}(C)$ so that $\psi(\vec d,c)$.  Hence, $\varphi(\vec d, \vec \ell)$ holds, i.e. $\vec \ell \cdot \vec d=0$ and $\vec d \neq 0$.  
Note that $C\subseteq L$ is a regular extension of fields (in characteristic 0) if and only if $L$ is linearly disjoint from $\acl_{\Ltilde}(C)$ over $C$.  So there must also be $\vec c\in C$ with $\vec \ell\cdot \vec c = 0$.
\end{proof}

\section{Separated bases}

The notion of a good separated basis was isolated in \cite{HHM2}, based on earlier observations by many different authors. In this section, we show that a field extension can often be assumed to have the separated basis property and that some type implications imply that the property can be lifted to a larger underlying field. In the subsequent section, we deduce strong consequences towards domination results from the separated basis property.  Many results in earlier papers on domination used the assumption that the base $C$ is maximal. Recall that a valued field is maximal (also called maximally complete or spherically complete) if it has no proper immediate extension. Here we show that this assumption can be replaced by the weaker assumption that there is  a good separated basis over $C$. 

\begin{defi}
Let $M$ be a valued field extension of $C$. Let $V\subseteq M$ be a $C$-vector space.  Let $m_1,\ldots, m_k$ be elements of $V$, $\vec{m}=(m_1,\ldots,m_k)$ and write $C\cdot \vec{m}$ for the
$C$-vector subspace of $V$ generated by $m_1,\ldots,m_k$. We say that $\{m_1,\ldots,m_k\}$ is a {\em separated basis over $C$}
if for all $c_1,\ldots, c_k$ in $C$, 
\[ v(\sum_{i=1}^k c_im_i ) = \min\{v(c_im_i): 1\le i\le k\}
\]
(and so, in particular, forms a basis for $C\cdot \vec{m}$).
We say that the separated basis is {\em good} if in addition for all $1\le i,j\le k$, either $v(m_i)=v(m_j)$ or $v(m_i)-v(m_j) \notin \Gamma_C$.
We say that $V$ has the {\em (good) separated basis property over $C$ } if every finite-dimensional $C$-subspace of $V$ has a (good) 
separated basis.
\end{defi}

By the next two lemmas, if the base $C$ is either maximal or trivially valued, then any field extension has the good separated basis property.

\begin{lemma}\cite[Proposition 12.1]{HHM2} \label{maximalimpliesgoodseparated}
Let $C$ be a non-trivially valued maximal field and $M$ a valued field extension. Then $M$ has the good separated basis property over $C$.
\end{lemma}

\begin{lemma}\label{trivially-valued-base}
Let $C$ be a trivially valued field, and $M$ a non-trivially valued field extension. Then $M$ has the good separated basis property over $C$. 
\end{lemma}

\begin{proof}
Since $v(c)=0$ for every $c\in C$, the condition for being good is vacuous. To construct separated bases, let $V$ be a finite-dimensional $C$-subspace of $M$, and proceed by induction on $\dim(V)$. If $\dim(V)=1$ then any basis is automatically separated. 

Assume the result is true for any $\ell$-dimensional subspace, and let $\{m_1,\ldots, m_\ell\}$ be a separated basis for $C\cdot \vec{m}$, the vector space that $\vec{m}=(m_1,\ldots,m_\ell)$ generates over $C$. Assume without loss of generality that $v(m_1)\le v(m_2)\cdots \le v(m_\ell)$. Notice that, for all $m\in\ C\cdot\vec{m}$, $v(m)\in\{v(m_1),\ldots, v(m_\ell)\}$. First suppose there is $m \in V\setminus C\cdot\vec{m}$ with $v(m)\notin\{v(m_1),\ldots, v(m_\ell)\}$. Then $\{m_1,\ldots,m_\ell,m\}$ is linearly independent and 
is separated. For suppose not. Then there are $c_1,\ldots,c_{\ell+1}$ such that $v(c_{\ell+1}m) = v(\sum_{i=1}^\ell c_i m_i)$. Since $v(c_{\ell+1}m) =v(m)$ and $v(\sum_{i=1}^\ell c_i m_i)\in \{v(m_1),\ldots, v(m_\ell)\}$, this contradicts the hypothesis on $m$.

Now suppose there is no such $m$.  Let $i_0$ be the greatest $i\le \ell$ for which there is $m\in V\setminus C\cdot \vec{m}$ with $v(m)=v(m_{i_0})$. We claim that $\{m_1,\ldots, m_\ell, m\}$ is a separated basis. Suppose not.  Then there are some $c_1,\ldots, c_\ell,c_{\ell+1}$ for which the valuation of the sum is not given by the minimum.  Write $I=\{i: v(m_i)=v(m_{i_0})\}$.   In particular we must have (by induction) 
\[  v(\sum_{i\in I}c_im_i + c_{\ell+1}m) > v(m_{i_0})
\]
and $c_{\ell+1}\neq 0$.  But then $\tilde m = \sum_{i\in I}c_im_i + c_{\ell+1}m$ must have valuation that is not among the valuations of $m_1, \dots, m_\ell$, or it must have valuation equal to $v(m_k)$ with $k>i_0$, which in either case contradicts our choice of $m$.

\end{proof}

\begin{prop}\label{dominationtoimmediateimpliesseparated}
Let $C$ be a field and $L$ a regular extension. Assume there is $F$  a maximal immediate extension of $\Calg$  such that $\tp_{\Ltilde}(L/C)\vdash \tp_{\Ltilde}(L/F)$.  Then $L$ has the good separated basis property over $C$.  Moreover, if $C'$ is any algebraically closed field  with $C\subseteq C' \subseteq F$, then the $C'$-vector space generated by $L$ inside $LF$ also has the good separated basis property over $C'$.
\end{prop}

\begin{proof} 
If $C$ is trivially valued, then the conclusion follows immediately from Lemma \ref{trivially-valued-base}. So assume that $C$ is not trivially valued.

The proof is by induction on the dimension of a finitely generated vector subspace of $L$ over $C$. The base case is immediate, so assume for the induction hypothesis that $\ell_1,\ldots,l_{n+1}$  are linearly independent over $C$ and that $\vec \ell = (\ell_1, \dots, \ell_n)$ is a good separated basis not only for the space it generates over $C$ but also for the space it generates over any algebraically closed $C'$ with $C\subseteq C'\subseteq F$. By Lemma~\ref{unique extension implies linearly disjoint}, $\ell_1,\ldots,l_{n+1}$ are linearly independent over $F$. As $F$ is maximal (see the claim in the proof of \cite[Proposition 12.1]{HHM2}), there is a closest element of $F\cdot \vec\ell$ to $\ell_{n+1}$; say 
\[
	v(\sum_{i=1}^n b_i\ell_i - \ell_{n+1})= \gamma
\]
realizes this maximal valuation.  Note that $\Gamma_F=\Gamma_{\Calg}$ by choice of $F$ and that $\Gamma(C)=\dcl_{\Ltilde}(C)\cap \Gamma =\Gamma(\Calg)$.  Thus we may apply Lemma \ref{Gamma(LF) is Gamma(L)} to see that $\Gamma(LF) =\Gamma(L)$, and hence $\gamma\in\Gamma(L)$. 
In fact, applying Lemma \ref{Gamma(LF) is Gamma(L)} with $L$ replaced by $L_0= C(\ell_1, \dots, \ell_{n+1})$ one sees that $\gamma \in \dcl_{\Ltilde}(C(\ell_1, \dots, \ell_{n+1}))$.

\begin{claim}
There is $b'\in \Calg\cdot\vec\ell$ with $v(b'-\ell_{n+1})=\gamma$. 
\end{claim}

\begin{proof of claim}
Let $k=\trdeg(b_1,\ldots,b_n/C)$, assume that $k$ is the minimum transcendence degree of any tuple in $\vec{d}\in F$ such that $v(\vec{d}\cdot\vec{\ell} -\ell_{n+1}) = \gamma$ and assume for contradiction that $k\ge 1$. Fix an algebraically closed $C'\subseteq C(b_1,\ldots,b_n)^{\mathrm alg}$ such that 
$\trdeg(C'(b_1,\ldots,b_n)/C')=1$ and, without loss of generality, assume that $b_1\notin C'$, that $b_2, \dots, b_k\in C'$ are algebraically independent over $C$, and that $\psi(b_1, \dots, b_k, x_{k+1}, \dots, x_n)$ is a formula which holds of $b_{k+1}, \dots, b_n$ and implies the algebraicity of $b_{k+1}, \dots, b_n$ over $C, b_1, \dots, b_k$.

Note that $b_1$ is also transcendental over $C'L$. For, 
 since $\tp(L/C')\vdash\tp(L/F)$, we have that $\tp(LC'/C')\vdash\tp(LC'/F)$ and so  $\tp(F/C')\vdash \tp(F/C'L)$ by Lemma~\ref{type implications}.  Hence if $b_1$ were algebraic over $C'L$, it would also be algebraic over $C'$, which it is not.

Let $\varphi(x_1)$ be the formula over $C'\ell_1\dots \ell_{n+1}$ 
\[\exists x_{k+1}\dots \exists x_n (v(x_1\ell_1+\sum_{i=2}^k b_i\ell_i +\sum_{i=k+1}^n x_i\ell_i- \ell_{n+1} )= \gamma ) \land \psi(x_1, b_2, \dots, b_k, x_{k+1}, \dots, x_n).
\] 
Since $\varphi(b_1)$ holds, and $b_1$ is not algebraic over $C'\vec\ell \ell_{n+1}$, we may assume that $\varphi(x_1)$ defines a finite union of $\acl_{\Ltilde}(C'\vec\ell \ell_{n+1})$-definable swiss cheeses. Suppose for contradiction the swiss cheese containing $b_1$ does not intersect $C'$. 

First note that all the points contained in it have the same $\Ltilde$-type over $C'$. For suppose not. Then the outer ball of the swiss cheese contains a $C'$-definable closed ball of radius $\beta$. This closed ball contains infinitely many points of $C'$ of distance $\beta$ apart, which therefore cannot all be contained in the excluded balls of the swiss cheese, and hence at least one satisfies $\varphi$.  It follows in particular that all extensions of $C'$ generated by an element of this swiss cheese are isomorphic over $C'$.

There is  
$d\in \acl_{\Ltilde}(C'\vec\ell \ell_{n+1})$ realizing $\varphi(x_1)$, since this is a model. Since $\tp(d/C')=\tp(b_1/C')$ and $\tp(b_1/C')\vdash \tp(b_1/C'L)$, we have that $\tp(d/C'L)=\tp(b_1/C'L)$. 
 However, the extension $C'(d)$,  cannot be isomorphic over $C'\vec\ell \ell_{n+1}$ to $C'(b_1)$, as $b_1$ is transcendental over $C'(\vec\ell \ell_{n+1})$. 
 
 This contradiction shows that there is $b'_1\in C'$ realizing $\varphi(x_1)$ and hence also $b'_{k+1},\ldots,b'_n$ such that 
 \[v(b'_1\ell_1+\sum_{i=2}^k b_i\ell_i +\sum_{i=k+1}^n b'_i\ell_i- \ell_{n+1} )= \gamma 
\]
Since $\psi(b'_1, b_2, \dots, b_k, x_{k+1}, \dots, x_n)$ holds of $b'_{k+1},\ldots,b'_n$, it follows that $b'_{k+1},\ldots,b'_n \in C'$.   Thus $b_1',b_2,\ldots,b_k,b'_{k+1},\ldots,b'_n$ is a tuple in $C'$ which witnesses the contradiction with the definition of $k$. 
\end{proof of claim}

\begin{claim}
There is $b''\in C\cdot \vec\ell $ with $v(b''-\ell_{n+1})=\gamma$. 
\end{claim}
\begin{proof of claim}
We have $b' = \sum_{i=1}^n b'_i\ell_i \in\Calg \cdot \vec\ell$ with $v(b'-\ell_{n+1})=\gamma$. Let $\Aut(\Calg/C)$ act on $b'_1,\ldots, b'_n$  and let 
$b^1=b',\ldots,b^{m}$ be the conjugates of $b'$ under this action. As $\tp_{\Ltilde}(L/C)\vdash \tp_{\Ltilde}(L/\Calg)$ by assumption, and thus $\tp(\Calg/C)\vdash \tp(\Calg/L)$, we have that for every $j<m$, $v(b^j-\ell_{n+1})=\gamma$. Let $b''=\frac 1 m \sum_{j\le m} b^j$ (using the equicharacteristic 0 assumption). Then
\[
    v(b''-\ell_{n+1}) = v(\frac 1 m \sum_{j\le m} (b^j - \ell_{n+1})) = \min_{j\le m}\{ v(b^j - \ell_{n+1})\} = \gamma
\]
as the valuation cannot be greater than $\gamma$, by its definition.
\end{proof of claim}

Now the argument is a straightforward calculation, as in  \cite[Proposition 12.1]{HHM2}. We let
$\ell'_{n+1} = \ell_{n+1}-b''$. Then $(\vec\ell,\ell'_{n+1})$ is a separated basis for the space it generates over $F$ and hence
also for the space generated over any subset of $F$, in particular for any $C'$ with $C\subseteq C'\subseteq F$. Then, as in \cite[Lemma 12.2]{HHM2}, the basis can be made into a good separated basis.
\end{proof}

As a corollary, we can show that the good separated basis property follows from stable domination. In the next section, we will prove that this characterizes stable domination. 

\begin{cor}\label{stablydominatedimpliesseparated}
Let $a$ be a tuple of valued field elements, let $C$ be a subfield of $\monster$, and suppose  that $L=\dcl_{\Ltilde}(Ca)$ is a regular extension of $C$. If $\tp_{\Ltilde}(a/C)$ is stably dominated then $L$ has the good separated basis property over $C$.
\end{cor}

\begin{proof}
Working in $\widetilde\monster$, let $F$ be any immediate extension of $\Calg$. Because $\St_C(F)=\St_C(C^{\alg})$ and thus $\St_C(F)\subseteq \acl_{\Ltilde}(\St_C(C))$ we have $\St_C(L)\ind_C \St_C(F)$. Because $\tp_{\Ltilde}(L/C)$ is stably dominated (and Proposition  \ref{resdominationequivalent}), 
$\tp_{\Ltilde}(L/C\St_C(F))\vdash \tp_{\Ltilde}(L/F)$. Clearly, $\tp_{\Ltilde}(L/\Calg) \vdash \tp_{\Ltilde}(L/C\St_C(F))$ and as $\tp_{\Ltilde}(L/C)\vdash \tp_{\Ltilde}(L/\Calg)$ by Lemma~\ref{orthogonaltoacl}, we have $\tp_{\Ltilde}(L/C)\vdash \tp_{\Ltilde}(L/F)$.  If $F$ is also maximal then we are in the situation of Proposition~\ref{dominationtoimmediateimpliesseparated}.
\end{proof}

The following lemma is stated as a claim in the proof of Proposition~12.11 of \cite{HHM2} and the subsequent lemma is part of the statement of that proposition.  However, in \cite{HHM2}, $C$ is assumed to be maximal. We repeat the proofs here in order to clarify that the maximality of $C$ is only used to obtain a separated basis.

\begin{lemma}\label{separatedbasislifts}
Let $L$, $M$ be valued fields with $C\subseteq L\cap M$ a valued subfield. Assume that $\Gamma_L\cap\Gamma_M  = \Gamma_C$, $k_L$ and $k_M$ are linearly disjoint over $k_C$, and $L$ has the good separated basis property over $C$. Choose $\{\ell_1, \dots \ell_k\}$ a good separated basis for the subspace of $L$ it generates over $C$.  Then  
$\{\ell_1,\ldots, \ell_k\}$ is still a good separated basis for the subspace 
 of $LM$ that it generates over $M$.
\end{lemma}

\begin{proof}
Suppose, for a contradiction, that there are $m_1,\ldots,m_k$  in $M$ such that 
\[
	v(\sum_{i=1}^k \ell_i m_i ) >  \min\{v(\ell_i m_i): 1\le i\le k\} = \gamma.
\]
Let $I\subseteq\{1,\ldots,k\}$ be the set of indices with $v(\ell_i m_i)=\gamma$ for $i\in I$. Note that $|I| > 1$ and for all $i, j$ in $I$, 
$v(\ell_i)-v(\ell_j) = v(m_j)-v(m_i) \in \Gamma_L\cap\Gamma_M = \Gamma_C$. Thus $v(\ell_i)=v(\ell_j)$ as the
basis is good. Fix $j\in I$ and write $I'=I\setminus\{j\}$. Now
\[
	v(\sum_{i\in I} \ell_i m_i) > \gamma \Longrightarrow v(1 + \sum_{i\in I'} \frac{\ell_i m_i}{\ell_j m_j}) > 0 
\]
and hence $\res(1 + \sum_{i\in I'} \frac{\ell_i m_i}{\ell_j m_j}) = 0$. As $v(\ell_i/\ell_j)=v(m_i/m_j)=0$, the residue map is
a ring homorphism, and hence 
\[
	1 + \sum_{i\in I'} \res(\ell_i/\ell_j)\res(m_i/m_j) =0.
\]
As $k_L$, $k_M$ are linearly disjoint over $k_C$, there must be $c_i\in C$ for $i\in I'$ with $\res(c_i)$ not all zero such that
$\res(c_j)+\sum_{i\in I'} \res(\ell_i/\ell_j)\res(c_i) =0$. Lifting back to the field gives $v (\sum_{i\in I'} \ell_i c_i ) > v(\ell_j)$,
which contradicts the assumption that $\{\ell_i:i\in I\}$ is separated over $C$. The basis is clearly good, as the value groups of $L$ and $M$ are disjoint over the value group of $C$.
\end{proof}

Lemma~\ref{separatedbasislifts} gives the following  purely algebraic statement. 

\begin{prop}\label{linearlydisjoint}
Let $L$, $M$ be valued fields with $C\subseteq L\cap M$ a valued subfield. Assume that $\Gamma_L\cap\Gamma_M  = \Gamma_C$,
that $k_L$ and $k_M$ are linearly disjoint over $k_C$ and that $L$ or $M$ has the good separated basis property over $C$. Then
$L$ and $M$ are linearly disjoint over $C$, $\Gamma_{LM}$ is the group generated by $\Gamma_L$ and $\Gamma_M$ over 
$\Gamma_C$ and $k_{LM}$ is the field generated by $k_L$ and $k_M$ over $k_C$.
\end{prop}

\begin{proof}Without loss of generality, $L$ has the good separated basis property over $C$. To prove the linear disjointness, it suffices to show that any finite tuple $\ell_1,\ldots , \ell_k$ from $L$ which is linearly independent 
over $C$ is also linearly independent over $M$ (recall that we are working inside some ambient structure, so this statement makes 
sense).  This follows from the conclusion of Lemma~\ref{separatedbasislifts}. 

Now let $x$ be in the ring generated by $L$ and $M$ over $C$. Then $x=\sum_{i=1}^k \ell_i m_i$ for some $\ell_i\in L$, $m_i\in M$ and we may assume that the $\ell_i$
form a good separated basis for the $C$-vector subspace of $L$ that they generate. By Lemma~\ref{separatedbasislifts} the tuple is also separated over
$M$ and hence $v(x) = v(\ell_j) + v(m_j)$ for some $j\in \{1,\ldots,l\}$. Thus 
$\Gamma_{LM} = \Gamma_L\oplus_{\Gamma_C} \Gamma_M$. Suppose that $\res(x)\ne 0$. Let $I=\{ i : v(\ell_i m_i) =0\}$.
Then $\res(x) =\res(\sum_{i\in I} \ell_i m_i) = \sum_{i\in I} \res(\ell_im_i)$, and hence the residue field of $k_{LM}$ is 
generated by $k_L$ and $k_M$.  
\end{proof}

\section{Preliminary domination results}

In this section, we show that a separated basis is strong enough to imply statements which are almost residue field domination results. The conclusion of Proposition~\ref{basicisomorphismtheorem} is not quite the statement of residue field domination for two reasons. Firstly, the type implication should be over the residue field of $M$, rather than the residue field of $L$. This is addressed in Corollary~\ref{revisedbasiclemma}. Secondly, the type implication needs to proved for subsets of any sort, not just the field sort. This is addressed in Section~4.
 
The first proposition shows that the good separated basis property is exactly what is needed in order to show type implication. The first part is a statement about $\widetilde \monster$ and is Proposition~12.11 of \cite{HHM2}, except with the assumption of a good separated basis replacing the maximality of $C$.  The further conclusion of this proposition is proved in 
\cite[Theorem~2.5]{EHM} in the case of real closed valued fields. The proof given here is very similar, and illuminates the
key properties to verify that the isomorphism of valued fields is actually an isomorphism 
of the full structure.

\begin{prop}\label{basicisomorphismtheorem}
Let $L$, $M$ be valued fields with $C\subseteq L\cap M$ a valued subfield. Assume that $\Gamma_L\cap\Gamma_M  = \Gamma_C$,
that $k_L$ and $k_M$ are linearly disjoint over $k_C$ and that $L$ or $M$ has the good separated basis property over $C$. Let 
$\sigma:L\to L'$ be a valued field isomorphism which is the identity on $C$, $\Gamma_L$ and $k_L$. Then $\sigma$ extends
by the identity on $M$ to a valued field isomorphism from $LM$ to $L'M$, and thus $\tp_{\Ltilde}(L/Ck_L \Gamma_L)\vdash \tp_{\Ltilde}(L/M)$.

Suppose further that $L$ and $M$ are substructures of $\mathcal{U}$ and $\sigma$ is an $\calL$-isomorphism. Then $\sigma$ is an isomorphism of $\RV_{LM}$ to $\RV_{L'M}$, and thus $\tp(L/Ck_L \Gamma_L)\vdash \tp(L/M)$.

\end{prop}

\begin{proof}
By Lemma~\ref{linearlydisjoint},  $L$ and $M$  are linearly disjoint over $C$.  Since $k_L'=k_L$, $\Gamma_L'=\Gamma_L$, and $L'$ has the good separated basis property over C whenever $L$ does, Lemma~\ref{linearlydisjoint} also implies that $L'$ and $M$ are linearly disjoint over $C$.  Hence $\sigma$ extends to a field isomorphism on $LM$ given by $\sigma(\sum\ell_i m_i) = \sum \sigma(\ell_i) m_i$ for any $\ell_i\in L$, $m_i\in M$. 

To show that $\sigma$ preserves the valuation on $LM$, choose $x$ in the ring generated by $L$ and $M$ over $C$ and write $x= \sum_{i=1}^k \ell_i m_i$.  First suppose that $L$ has the good separated basis property over $C$. We may assume that $\{\ell_1,\ldots,\ell_k\}$
is separated over $C$ and, as $\sigma$ is a
valued field isomorphism on $L$, this implies also that $\{\sigma(\ell_1),\ldots,\sigma(\ell_k)\}$ is separated over $C$.   Hence, by Lemma~\ref{separatedbasislifts}, both bases are separated over $M$.  Then
\[
	v(x) = \min_{1\le i \le k}\{v(\ell_i) + v(m_i)\} =  \min_{1\le i \le k}\{v(\sigma(\ell_i)) + v(m_i)\}  = v(\sigma(x)),
\]
as required.  On the other hand, if we suppose that $M$  has the good separated basis property over $C$, we may assume that $\{m_1,\ldots,m_k\}$
is separated over $C$ and hence, by  Lemma~\ref{separatedbasislifts}, separated over $L$ and $L'$.   Then, as before,
\[
	v(x) = \min_{1\le i \le k}\{v(\ell_i) + v(m_i)\} =  \min_{1\le i \le k}\{v(\sigma(\ell_i)) + v(m_i)\}  = v(\sigma(x)),
\]
as required.

Note that $\sigma$ is the identity on $k_L$ and $k_M$ and hence by Lemma \ref{linearlydisjoint} on $k_{LM}$.  Likewise, it is the identity on $\Gamma_{LM}$.  Since $\sigma:LM\to L'M$ is a valued field isomorphism, it automatically preserves the group structure on $\RV_{LM}$.  Hence, to show that $\sigma:\RV_{LM} \to \RV_{L'M}$ is an isomorphism it suffices, by the quantifier elimination result in Proposition \ref{qe} and the fact that $\sigma$ is the identity on $\Gamma_{LM}$ and $k_{LM}$, to prove that 
the $n$th power predicates are preserved; that is, $P_n(\rv(a)) \Longleftrightarrow P_n(\sigma(\rv(a))$. For each $n$, we 
have assumed there is a finite set of constants $\{\lambda\}$ which are representatives for the cosets of $P_n$. Of course,
$\sigma(\lambda_L) = \lambda_{L'}$. Consider a coset representative $\rho$.  Since for any $x, y \in RV$, whether or not $xy$ is in the same coset as $\rho$ depends only on the coset of $x$ and the coset of $y$, we have for each $\rho$ a finite set of pairs $\Lambda_{\rho,n}=\{(\lambda,\mu)\}$ such that
 
\[ P_n(\rho^{-1}xy) \Longleftrightarrow \bigvee_{(\lambda,\mu)\in \Lambda_{\rho,n}} P_n(\lambda x) \,\&\, P_n(\mu y).
\]
\begin{claim} Suppose $a= \ell m$ for some $\ell\in L$, $m\in M$. Then for every $n$, 
$P_n(\rho^{-1}\sigma(\rv(a)))  \Longleftrightarrow P_n(\rho^{-1}\rv(a))$. 
\end{claim}

\begin{proof of claim} 
\begin{align*}
	P_n(\rho^{-1}\rv(a)) & \Longleftrightarrow \bigvee_{(\lambda,\mu)\in \Lambda_{\rho,n}} P_n(\lambda \rv(\ell)) \,\&\, P_n(\mu \rv(m)) \\
		& \Longleftrightarrow \bigvee_{(\lambda,\mu)\in \Lambda_{\rho,n}} P_n(\sigma(\lambda \rv(\ell)) \,\&\, P_n(\sigma(\mu \rv(m)), \mbox{ as $\sigma | L$ is an isomorphism and $\sigma| M =\id$}; \\
		& \Longleftrightarrow \bigvee_{(\lambda,\mu)\in \Lambda_{\rho,n}} P_n(\lambda \rv(\sigma(\ell))) \,\&\, P_n(\mu \rv(m)) \\
		& \Longleftrightarrow P_n(\rho^{-1}\sigma(\rv(a)) .
\end{align*}
\end{proof of claim}
Now let $a=\sum_{i=1}^n \ell_i m_i$ for some $n>1$. By Lemma~\ref{linearlydisjoint}, $v(a)$ is in the group 
generated by $\Gamma_L$ and $\Gamma_M$, so there are $\ell\in L$ and $m\in M$ with $v(a) = v(\ell m)$. Write 
$a= \ell m a_0$, where $v(a_0) = 0$ and note that $a_0\in LM$.  Then $\rv(a_0) = \res(a_0)$. As 
$\sigma$ is the identity on $k_{LM}$, $\sigma(\res(a_0)) = \res(a_0)$, and therefore  $\sigma(\rv(a_0)) = \rv(a_0)$.Thus 
$P_n(\sigma(\rv(a_0)) \Longleftrightarrow P_n(\rv(a_0))$. Hence
\begin{align*}
	P_n(\rv(a)) & \Longleftrightarrow \bigvee_{(\lambda,\mu)\in \Lambda_{1,n}} P_n(\lambda \rv(\ell m)) \,\&\, P_n(\mu \rv(a_0)) \\
		& \Longleftrightarrow \bigvee_{(\lambda,\mu)\in \Lambda_{1,n}} P_n(\lambda \sigma(\rv(\ell m))) \,\&\, P_n(\mu \sigma(\rv(a_0))), \mbox{by the claim and the above}; \\
		& \Longleftrightarrow P_n(\sigma(\rv(a)) .
\end{align*}
\end{proof}

As in \cite{EHM}, it is helpful to state the following corollary, which means in particular that we can change the hypothesis on $\sigma$ to assume that it fixes the value group and residue field of $M$ instead of those of $L$.

\begin{cor}\label{revisedbasiclemma}
Let $L$, $M$ be substructures of $\monster$ with $C\subseteq L\cap M$ a valued subfield. Assume that 
$\Gamma_L\cap\Gamma_M = \Gamma_C$, that $k_L$ and $k_M$ are linearly disjoint over $k_C$, and that $L$ or $M$ has the good separated basis property over $C$. Then  $\tp(L/C\Gamma_M k_M)\vdash \tp(L/M) $. Similarly, if $L$ and $M$ are substructures of $\widetilde\monster$ satisfying the same hypotheses, then $\tp_{\Ltilde}(L/C\Gamma_M k_M)\vdash \tp{\Ltilde}(L/M) $.
\end{cor}

\begin{proof}
By Proposition~\ref{basicisomorphismtheorem}, we have $\tp(L/C\Gamma_L k_L)\vdash \tp(L/M)$.  Applying $(v)\Rightarrow (ii)$ of Proposition~\ref{resdominationequivalent}, we obtain  $\tp(L/C\Gamma_M k_M)\vdash \tp(L/M) $.
\end{proof}

\begin{rem}\label{residue field domination in the field sort}
If, in the preceding corollary, $L$ could be taken from any sort, we would have proven the following: if $k(M)$ is a regular extension of $k(C)$, $\Gamma_M=\Gamma_C$, and $M$ has the good separated basis property over $C$ then $\tp(M/C)$ is residue field dominated.
\end{rem}

Corollary \ref{revisedbasiclemma} often has implications for how forking behaves.  When $T$ is such that forking and dividing are the same, Corollary \ref{revisedbasiclemma} describes circumstances in which forking in $\monster$ can be reduced to  forking in the residue field and value group, which is presumably easier to understand.

\begin{cor}\label{forking}
Assume that $T$ implies that forking and dividing are the same over $C$, and assume further that $k(Ca)$ is a regular extension of $k_C$, $\Gamma(Ca)/\Gamma_C$ is torsion free, and either $\dcl(Ca)$ or $\dcl(Cb)$ has the good separated basis property over $C$.  Then $a \ind_C b$ if and only if $k(Ca)\Gamma(Ca)\ind_C k(Cb)\Gamma(Cb)$.
\end{cor}

\begin{proof}
The proof is exactly that of Lemma 3.3 (i) and Theorem 3.4 (ii) of \cite{EHM}, with the reference to Corollary 2.8 of that paper replaced by Corollary \ref{revisedbasiclemma} of this one, and the use of elimination of imaginaries in the residue field replaced by strong stable embeddedness as in Remark \ref{stably embedded}.
\end{proof}

As a further corollary, we give a purely algebraic characterization of  stable domination in ACVF (at least for a regular extension). We first note the following lemma.

\begin{lemma}\label{separated_bases_and_immediate_extensions}
Let $C,L$ be valued fields with $C\subseteq L$ and suppose that $L$ is  henselian and an unramified regular extension of $C$.  Then the following are equivalent:
\begin{enumerate}
    \item $L$ has the good separated basis property over $C$ 
    \item $\tp_{\Ltilde}(L/C) \vdash \tp_{\Ltilde}(L/F)$ for some maximal immediate extension $F$ of $\Calg$.
    \item $\tp_{\Ltilde}(L/C) \vdash \tp_{\Ltilde}(L/F)$ for any maximal immediate extension $F$ of $\Calg$.
\end{enumerate}
\end{lemma}
\begin{proof}  $(3) \Rightarrow (2)$ is clear and $(2) \Rightarrow (1)$ is Proposition~\ref{dominationtoimmediateimpliesseparated}. 

Let $F$ be any maximal immediate extension of $\Calg$ and assume that $L$ has the good separated basis property over $C$.  We apply Lemma \ref{separatedbasislifts} with $\Calg$ replacing $M$. The lemma applies because $L$ being henselian and regular implies that $k_L$ is a regular extension of $k_C$: for otherwise there would be a polynomial with coefficients in $k_C$ with a root in $k_L$, which would then lift to a polynomial over $C$ with a root in $L$ (as $L$ is henselian and the residue characteristic is zero), contradicting the regularity of $L$ over $C$. Applying Corollary \ref{revisedbasiclemma}, with $L\Calg$ playing the role of $L$, $\Calg$ playing the role of $C$, and $F$ playing the role of $M$, we see that $\tp_{\Ltilde}(L\Calg/\Calg) \vdash \tp_{\Ltilde}(L\Calg/F)$, and hence $\tp_{\Ltilde}(L/\Calg) \vdash \tp_{\Ltilde}(L/F)$.  Now apply Lemma \ref{orthogonaltoacl} to obtain that $\tp_{\Ltilde}(L/C) \vdash \tp_{\Ltilde}(L/F)$.
\end{proof}

\begin{theorem}\label{domination-equivalence}
Suppose that $\monster$ is algebraically closed. Let $C \subset \monster$ be a subfield, let $a$ be a tuple of valued field elements, and let $L$ be the definable closure of $Ca$ in the valued field sort.  Assume $L$ is a regular extension of $C$. Then the following are equivalent.
\begin{enumerate}[(i)]
    \item $\tp_{\Ltilde}(a/C)$ is stably dominated.
    \item $L$ has the good separated basis property over $C$ and $L$ is an unramified extension of $C$.
\end{enumerate} 
\end{theorem}

\begin{proof}
First assume (ii).  Since $L$ is definably closed, it is henselian.  Thus we may apply Proposition~\ref{separated_bases_and_immediate_extensions} to see that $\tp_{\Ltilde}(L/C)\vdash\tp_{\Ltilde}(L/F)$ for some maximal extension $F$ of $\Calg$. Applying Lemma \ref{linearlydisjoint}, we see that $\Gamma_{L\Calg}=\Gamma_{\Calg}$. It follows that $\Gamma({L\Calg})=\Gamma({\Calg})$, as both are equal to $\Gamma_{\Calg}$. By \cite[Proposition 12.5]{HHM2}, it follows that $\tp(a/\Calg)$ is orthogonal to $\Gamma$, which by Fact \ref{stably dominated implies orthogonal} is equivalent to being stably dominated.  By Fact \ref{stable domination lifts to acl}, $\tp_{\Ltilde}(a/C)$ is stably dominated as $\tp_{\Ltilde}(a/\Calg)$ is stably dominated.

The converse is handled by Corollary \ref{stablydominatedimpliesseparated} along with the fact that stable domination implies orthogonality to the value group.
\end{proof}

\subsection{RV-domination}
As we recalled in Example~\ref{stably not residue field dominated}, stable domination over the value group in an algebraically closed valued field \cite[Theorem 12.18]{HHM2} is implied by the assumptions that the base $C$ is maximal, $k(L)$ is a regular extension of $k(C)$, and $\Gamma_L/\Gamma_C$ is torsion-free.  We have already noted that this is not enough to get residue field domination over the value group. Here we introduce a notion of $\RV$-domination, a property which does hold for the above example, and which in some ways feels closer to stable domination.

The analogue to the stable part of an algebraically closed valued field is here given by an infinite collection of definable subsets of $\RV$, each of which is internal to the residue field. Let $M\supseteq C$ and $S\subset \Gamma$.  Recall that  $\RV_\gamma(M)$ is the fiber of the valuation map in $\RV(M)$ above $\gamma$, for $\gamma\in S$.  
Although this might seem to be very different from $\St_C(M)$, in fact, by 
 \cite[12.9]{HHM2}, when $C$ and $M$ are algebraically closed and $S$ is definably closed, $\acl_{\Ltilde}(\{\RV_\gamma(M)\}_{\gamma \in S})$ is essentially $\St_{CS}(M)$.  Furthermore, \cite[12.10]{HHM2}) gives  equivalent conditions for independence over $C\Gamma_L$ of $\St_{C\Gamma_L}(L)$ and $\St_{C\Gamma_L}(M)$. We take one of these equivalent conditions and use it as the definition of algebraic independence in $\RV$.

\begin{defi}\label{RV independence}
    Let $L$, $M$ be subfields of $\monster$ with $C\subseteq L\cap M$ a valued subfield. Assume that $\Gamma_L\subseteq \Gamma_M$ and $\Gamma_L/\Gamma_C$ is torsion-free.  We say that $\{ \RV_\gamma(L)\}_{\gamma\in\Gamma_L}$ is {\em algebraically independent} from $\{ \RV_\gamma(M)\}_{\gamma\in\Gamma_L}$ over $C\Gamma_L$ if the following condition holds: for every sequence $(a_i)$, $(b_i)$ of elements of $L$, and $(e_i)$ of elements of $M$ such that 
    \begin{description}
    \item[-]
    $(v(a_i))$ is a $\mathbb{Q}$-basis for $\Gamma(L)$ over $\Gamma(C)$, 
    \item[-] $(\res(b_i))$ is a transcendence basis of $k_L$ over $k_C$, and 
    \item[-] for all $i$, $v(a_i) = v(e_i)$, 
    \end{description}
     the sequence $(\res(a_{i}/e_{i}), \res(b_{j}))$ is algebraically independent over $k(M)$ .
\end{defi}

\begin{defi}\label{RV domination}
Let  $C\subseteq L$ be subfields of $\monster$ such that  $\Gamma_L/\Gamma_C$ is torsion-free.  We say $\tp(L/C\Gamma_L)$ is {\em $\RV$-dominated} if for any subfield $M\supseteq C$ such that $\Gamma_M\supset\Gamma_L$, if  $\{ \RV_\gamma(L)\}_{\gamma\in\Gamma_L}$ is algebraically independent from $\{ \RV_\gamma(M)\}_{\gamma\in\Gamma_L}$ over $C\Gamma_L$ then 
\[ \tp(M/C\{\RV_\gamma(L)\}_{\gamma\in\Gamma_L})\vdash \tp(M/L). \]
\end{defi}

\noindent We note that this is not quite domination by RV, which is not a stable set in an algebraically closed valued field, but rather domination by a collection of $k$-internal sets.  However, the more accurate name "$\RV_\gamma$ where $\gamma$ ranges over $\Gamma_L$  domination" is too unwieldy.

In order to prove a domination theorem, we first prove a result about extending isomorphisms. The following theorem was  originally given in \cite[Proposition 12.15]{HHM2} in the case of algebraically closed valued fields, and then in \cite[Theorem 2.9]{EHM} for real closed valued fields. The proof is somewhat subtle, and it is not completely obvious that the changes that are required for the current, more general, context will carry through the machinery. For this reason, we repeat the proof in this paper, but postpone it to the appendix.

\begin{theorem}\label{HHM 12.15}
Let $L$, $M$ be subfields of $\Utilde$ with $C\subseteq L\cap M$ a valued subfield, $k(L)$ a regular extension of $k(C)$, and $\Gamma_L/\Gamma_C$ torsion-free. Assume that $\Gamma_L \subseteq \Gamma_M$, that $\{\RV_\gamma(L)\}_{\gamma\in\Gamma_L}$ is algebraically independent from $\{\RV_\gamma(M)\}_{\gamma\in\Gamma_L}$ over $C\Gamma_L$ and that $L$ has the good separated basis property over $C$. Let $\sigma$ be an automorphism of $\Utilde$ mapping $L$ to $L'$ which is the identity on $C$, $\Gamma_L$, and $k_M$.
Then $\sigma|_L$ can be extended to a valued field isomorphism from $LM$ to $L'M$ which is the identity on $M$. Furthermore,  if $\sigma$ is additionally the identity on $\RV_L$, then $\sigma$ may be extended to $LM$ so that it is the identity on $\RV_{LM}$.
\end{theorem}

\begin{theorem}\label{dominationovervaluegroup}
Let $L$, $M$ be subfields of $\mathcal{U}$ with $C\subseteq L\cap M$ a valued subfield, $k(L)$ a regular extension of $k(C)$, $\Gamma_L \subseteq \Gamma_M$ and $\Gamma_L/\Gamma_C$ is torsion-free. Assume that $\{\RV_\gamma(L)\}_{\gamma\in\Gamma_L}$ is algebraically
independent from $\{\RV_\gamma(M)\}_{\gamma\in\Gamma_L}$ over $C\Gamma_L$ and that $L$ has the good separated basis property over $C$. 
Let $\sigma:L\to L'$ be an $\calL$-isomorphism which is the identity on $C$, $\{\RV_\gamma(M)\}_{\gamma\in\Gamma_L} $.
Then $\sigma$ can be extended by the identity on $M$ to an automorphism of $\mathcal{U}$.
\end{theorem}

\begin{proof}
We wish to show that $\tp(L/C\{\RV_\gamma(M)\}_{\gamma\in\Gamma_L} )$ implies $\tp(L/M)$. Observe that for each $\gamma\in \Gamma_L$, both $\RV_\gamma(L)$ and $\RV_\gamma(M)$ are nonempty. This (by Remark \ref{stably embedded}) allows us to apply $(iv) \Rightarrow (i)$ of Proposition~\ref{resdominationequivalent}, and we see that it suffices to show that 
\[\tp(L/C\{\RV_\gamma(L)\}_{\gamma\in\Gamma_L} \{\RV_\gamma(M)\}_{\gamma\in\Gamma_L} )\vdash\tp(L/M).
\]
The assumption that $\sigma$ fixes $\{\RV_\gamma(M)\}_{\gamma\in\Gamma(L)}$ implies that $\sigma$ fixes $k_M$ and $\Gamma_L$. 
By the above, we may assume that $\sigma$ fixes $\{\RV_\gamma(L)\}_{\gamma\in\Gamma_L}$ as well.
Thus we may apply Theorem~\ref{HHM 12.15}, to get a valued field isomorphism $\sigma:LM\to L'M$ which is the identity on $M$ and on $\RV_{LM}$.  In order to show that $\sigma$ extends to an automorphism of $\mathcal{U}$, it suffices to show that it induces an isomorphism from the structure $\RV_{LM}$ to $\RV_{L'M}$, which is clear as the induced map is the identity.
\end{proof}

\begin{theorem}\label{RVdomination}
Let $L$ be a subfield of $\mathcal{U}$ with $C\subseteq L$ a valued subfield. Assume that $k(L)$ is a regular extension of $k(C)$, $\Gamma_L/\Gamma_C$ is torsion-free and that $L$ has the good separated basis property over $C$.  Then $\tp(L/C\Gamma_L)$ is $\RV$-dominated.
\end{theorem}

\begin{proof}
Let $M$ be a subfield of $\mathcal{U}$ as required in Definition~\ref{RV domination}. Theorem~\ref{dominationovervaluegroup} gives us that $\tp(L/C\{\RV_\gamma(M)\}_{\gamma\in\Gamma_L}) \vdash \tp(L/M)$.  As in the proof of Theorem~\ref{dominationovervaluegroup}, we may apply $(i) \Leftrightarrow (ii)$ of Proposition~\ref{resdominationequivalent} to obtain the type implication in the definition of $\RV$-domination.

\end{proof}

\section{The geometric sorts and domination}
 In the previous section, we worked within the field sort. However, our definition of residue field domination requires us to consider independent sets in any of the sorts. We thus need a mechanism to pull a hypothesis on an arbitrary geometric sort back to the field. This is given to us by the notion of a resolution. 

 The only sorts in $\monster$, apart from the main sort, are $\RV$ and $\Gamma$.  Of course, if one wanted to eliminate imaginaries, one would add more sorts including, but perhaps not limited to, the geometric sorts used to eliminate imaginaries in ACVF.  The results in this section, proven as they are by carrying out the arguments of \cite{HHM2} inside of $\monster$, apply also to the geometric sorts.  Thus for the remainder of this section, we will take $\monster$ to also refer that portion of $\monster^{eq}$ consisting of the geometric sorts. 
 
 \begin{defi}
 Let $A$ be a subset of $\monster$. We say that a set $B$ in the field sort is a {\em resolution } of $A$ if $B$ is algebraically closed (in the sense of $\mathcal{L}$) in the field sort and $A\subseteq\dcl(B)$. The resolution is {\em prime} if $B$ embeds over $A$ into any other resolution.
 \end{defi}
 
 In \cite[Theorem 11.14]{HHM2}, the existence of prime resolutions is shown for algebraically closed valued fields.  Thus given $A\subset \monster \subset \widetilde\monster$, we have a resolution $B\subseteq \widetilde \monster$, though it is not a priori clear that $B$ would be contained in $\monster$.  Below, we give a careful analysis  of the proof of the existence of resolutions, to see that the resolution can be constructed within $\monster$.  
Since the proof involves checking that the arguments of various parts of Chapter 11 of \cite{HHM2} never involve choosing something in $\widetilde\monster$ that necessarily lies outside of $\monster$, we follow the notation of \cite{HHM2} as we walk the reader through this process.  In particular, $K$ refers to the field sort and $R$ to the valuation ring.

\begin{theorem}\label{resolution}
Let $C \subseteq \monster$ be a subfield, and let $e\in \monster$ or more generally, in the geometric sorts of $\monster$. Then $Ce$ admits a resolution $B$ with $k(B)=k(\acl(Ce))$ and $\Gamma(B)=\Gamma(Ce)$. 

\end{theorem}

\begin{proof}
We follow the construction in Chapter 11 of \cite{HHM2}, with the notation there. First, as in Theorem 11.14, we can assume that $e=(a,b)$, where $a\in B_n(K)/B_n(R)$ and $b\in B_{m}(K)/B_{m,m}(R)$. The next step is to replace $e$ with an opaque layering of it (in the sense of ACVF). We need not concern ourselves here with the precise details of this, because we will follow the construction in Lemmas 11.10 to 11.13 exactly.  We need only check that the construction can be carried out in $\monster$ and does not require elements of $\widetilde \monster \setminus \monster$. Through multiple applications of Lemma 11.10 and Corollary 11.11, $a=gB_n(R)$ is replaced successively by pairs $(h(H\cap F),\ell(N\cap F^h))$, where $H,F$ are subgroups of $B_n(K)$, $N$ is a normal subgroup of $B_n(K)$, $h\in H$, $\ell\in N$. Those subgroups are some of the $G_i$ and $H_i$ defined in Lemma 11.12, and are defined over $\mathbb Z$. The decomposition asserted in that lemma holds over any ring, in particular it holds over our field $K(\monster)$. This shows that we can at each step take $h$ and $\ell$ in $K(\monster)$. The same is true for $b$.

So we have replaced $e$ by a sequence $\bar a=(a_0,\ldots,a_{N-1})$ satisfying the conditions of Lemma 11.4 in the sense of ACVF and lying in $\monster$. We therefore have $\dcl_{\widetilde \calL}(C\bar a)=\dcl_{\widetilde \calL}(Ce)$. Then we can find $C\subseteq D\subseteq K(\monster)$ such that $C\bar a\subset \acl_{\widetilde \calL}(D)$ and $D$ is atomic over $C\bar a$ (in $\widetilde \calL$). This is by Lemma 11.4: all we do is take representatives of the equivalence relations defining the $a_i$'s (here $D = B_0\cup C$ in the notation of Lemma~11.4). We can find such elements in $K(\monster)$ since $\bar a$ is in $\monster$.  Note that by the construction in Lemma 11.4, each representative is either in $D$ or algebraic over $D$.  In particular, each representative is contained in $\acl_{\widetilde \calL}(D)\cap K(\monster)$.

Next, we want to expand $D$ so that it remains atomic, but so that $C\bar a$ lies in the definable closure rather than the algebraic closure. We follow exactly the argument of Corollary 11.9, needing only to check that the construction does not leave $\monster$. We know that $\bar a$ is in the definable closure of some $b\in \acl_{\widetilde \calL}(D)\cap K(\monster)$ (namely the tuple of representatives).  The orbit (in the sense of $\widetilde \monster$) of $b$ over $D\bar a$ is finite, and hence coded by some $b'\in K(\widetilde{\monster})$. As $b'$ is definable over a subset of $\monster$, in particular $b'$ is in $K(\monster)$. We thus have $b'\in \dcl_{\widetilde \calL}(D\bar a)$ with $\bar a\in \dcl_{\widetilde \calL}(Db')$ and $\tp_{\Ltilde}(Db'/C\bar a)$ is isolated.   (Note that our $b$ is denoted $e$ in Corollary 11.9, and our $b'$ is denoted $e'$.) 

From Corollary 11.16, we know that $Ce$ admits a $\dcl$-resolution $B_0$ with $\dcl_{\Ltilde}(B_0)\cap k=\dcl_{\Ltilde}(Ce)\cap k$ and $\dcl_{\Ltilde}(B_0)\cap\Gamma=\dcl_{\Ltilde}(Ce)\cap\Gamma$. Referring to the proof of Corollary 11.16, we see that this $\dcl$-resolution is the one obtained in Corollary 11.9.  That is, $B_0=Db'$, with $D$ and $b'$ as above. Let $B=\acl(Db')\cap K(\monster)$. To see that $B$ is the required resolution, we just need to verify that $k(B)=k(\acl(Ce))$ and $\Gamma(B) = \Gamma(Ce)$.

First we show that $k(B_0)=k(Ce)$.  It is clear that $k(B_0) \supseteq k(Ce)$, so take $d\in k(B_0)$. witnessed by $\varphi$.  By quantifier elimination, $\varphi$ is an $\calL$-formula in the $\RV$-sort and has the form $\varphi(x, \rv(t(Db')))$ where $t$ is a term.  Since there are no additional terms in $\calL$ in the field sort, this is an $\Ltilde$ term, and thus $\rv(t(Db'))\in \dcl_{\Ltilde}(Db')$.    From the proof that $k$ is a stably embedded subset of $\RV$, we may assume $\rv(t(Db'))\in k$, and thus in $\dcl_{\Ltilde}(B_0)\cap k=\dcl_{\Ltilde}(Ce)\cap k$.  Thus $\varphi$ also witnesses that $d\in k(Ce)$.

Since it is clear that $k(B) \supseteq k(\acl(Ce))$, take $d_1\in\acl(B_0)\cap k$.  Suppose the conjugates of $d_1$ over $B_0$ are $d_1, \dots, d_n$.  Then the set $\{d_1, \dots, d_n\}$ is in the definable closure of $B_0$ and, as fields code finite imaginaries, the set is coded by an element of $k(B_0)=k(Ce)$.  Thus $d_1\in\acl(Ce)$, as desired.

A similar argument shows that $\Gamma(B)=\Gamma(\dcl_{\widetilde \calL}(Ce))$. 
\end{proof}

By the following lemma, we see that proving a type implication for such a resolution is sufficient to give us the desired type implication that we need in the definition of residue field domination. 

\begin{lemma}\label{domination and acl}
Fix a set of parameters $C$. Suppose that $B$ is a resolution of $Cb$ with $k(B)=k(\acl(Cb))$, and suppose that $\tp(a/Ck(B))\vdash \tp(a/CB)$.  Then $\tp(a/Ck(Cb))\vdash \tp(a/Cb)$.
\end{lemma}

\begin{proof}
Take $\varphi(x,b)\in \tp(a/Cb)$.  Since $b\in \dcl(B)$, there is $\psi(x,d_1)\in \tp(a/Ck(B))$ which implies $\varphi(x,b)$.  Consider the set $D=\{d_1, \dots, d_n\}$ of conjugates of $d_1$ over $Cb$.  This set is definable over $Cb$, and thus so is $\bigvee_{d_i \in D}\psi(x, d_i)$.  This latter formula is in $\tp(a/Ck(Cb))$ and implies $\varphi(x,b)$ as desired.
\end{proof}

The following lemma allows us  to assume that elements are in the main sort when trying to prove domination results..

\begin{lemma}\label{test with main sort}
Fix $\tp(a/C)$.  The following are equivalent:
\begin{enumerate}[(i)]
    \item for any $b\in \monster$, if $k(aC)\aclind_{k(C)} k(bC)$, then $\tp(b/C k(Ca))\vdash \tp(b/Ca)$. 
    \item for any $b$ in the field sort of $\monster$, if $k(aC)\aclind_{k(C)} k(bC)$, then $\tp(b/C k(Ca))\vdash \tp(b/Ca)$. 
\end{enumerate}
\end{lemma}

\begin{proof}
Clearly, (i) implies (ii).  For the other direction, assume that (ii) holds and choose $b\in \monster$ with $k(aC)\aclind_{k(C)} k(bC)$.  Choose a resolution $B$ of $Cb$ with $k(B)=k(\acl(Cb))$.  Since $k(aC)\aclind_{k(C)} k(B)$, we conclude by (2) that $\tp(B/Ck(Ca)\vdash \tp(B/Ca)$ and thus by the equivalence of (i) and (ii) in Proposition \ref{resdominationequivalent} that $\tp(a/C k(B))\vdash \tp(a/CB)$. Then we may apply Lemma \ref{domination and acl} to obtain $\tp(a/Ck(Cb))\vdash \tp(a/Cb)$.  We apply Proposition \ref{resdominationequivalent}  again to obtain $\tp(b/C k(Ca))\vdash \tp(b/Ca)$.
\end{proof}

As noted in Remark \ref{residue field domination in the field sort}, Lemma \ref{test with main sort} together with Corollary \ref{revisedbasiclemma} gives us the following residue field domination result.

\begin{theorem}\label{residue field domination for unramified extensions}
    Let $C\subseteq \monster$ be a subfield and let $a$ be a (possibly infinite) tuple of field elements such that the field generated by $Ca$ is an unramified extension of $C$ with the good separated basis property over $C$, and such that $k(Ca)$ is a regular extension of $k(C)$.  Then $\tp(a/C)$ is residue field dominated.
\end{theorem}

Using Theorem \ref{residue field domination for unramified extensions} (or rather its component pieces: Corollary \ref{revisedbasiclemma} and Lemma \ref{test with main sort}) we are able to push the above result a bit further and relate stable domination in the algebraically closed field to residue field domination in the henselian field. Recall that we write $C^+ = \acl(C)\cap \dcl(Ca)$.

\begin{theorem}\label{theorem:domfield}
Let $C\subseteq \monster$ be a subfield and let $a\in \monster$. Assume that $tp_{\widetilde \calL}(a/C)$ is stably dominated. Then $\tp(a/C^+)$ is residue field dominated.
\end{theorem}
\begin{proof} First assume that $a$ is a field element. By Fact~\ref{stable domination lifts to acl}, also 
 $\tp_{\widetilde \calL}(a/\acl(C))$ is stably dominated. Choose $b$ with $k(\acl(C)a)\aclind_{\acl(C)} k(\acl(C)b)$. By Lemma \ref{test with main sort}, we may assume that $b$ is a field element. Let $L$ be $\dcl(\acl(C)b)$ and let $M$ be $\dcl(\acl(C)a)$. Since $M$ is definably closed in $\mathcal{L}$ and thus also in $\Ltilde$, it is a henselian valued field, and trivially $M$ is a regular extension of $\acl(C)$, so we may use Corollary  \ref{stablydominatedimpliesseparated} to see that $M$ has the good separated basis property over $\acl(C)$. 
Note that $\Gamma_M=\Gamma_{\acl(C)}$ by stable domination, so trivially $\Gamma_L\cap\Gamma_M=\Gamma_{\acl(C)}$. Since $k(\acl(C)a)\aclind_{\acl(C)} k(\acl(C)b)$, Fact \ref{Lang} implies $k_L$ and $k_M$ are linearly disjoint over $\acl(C)$.
Thus Corollary \ref{revisedbasiclemma} implies that $\tp(b/\acl(C)k(\acl(C)a))  \vdash \tp(b/\acl(C)a)$ and hence $\tp(a/\acl(C))$ is residue field dominated. By Proposition~\ref{base acl}, $\tp(a/C^+)$ is residue field dominated.

Now let $a$ be in any of the sorts. 
By Fact~\ref{stable domination lifts to acl} and Fact~\ref{stably dominated implies orthogonal}
$\tp_{\widetilde \calL}(a/\acl(C))$ is orthogonal to $\Gamma$. By \cite[Lemma 10.14]{HHM2}, there is a resolution $B$ of $\acl(C)a$  such that $\tp(B/C)$ is orthogonal to $\Gamma$. On the other hand, we know by Theorem \ref{resolution} and \cite[Theorem 11.14]{HHM2}, that $\acl(C)a$ has a prime resolution $A$ that only adds algebraic elements to $k(Ca)$ and lies in $\monster$. By primality, $A$ embeds into $B$ and hence its $\Ltilde$-type is also orthogonal to $\Gamma$, so also stably dominated. 
 By Theorem \ref{theorem:domfield}, $\tp(A/C^+)$ is residue field dominated. Consider any $b\in \monster$ such that $k(C^+b)\aclind_{C^+} k(C^+a)$. Since $k(A)=\acl(k(C^+a))$, we have $k(C^+b)\aclind_{C^+} k(A)$.  By residue field domination for $\tp(A/C^+)$, we have $\tp(b/C^+ k(A))\vdash \tp(b/C^+A)$. Now apply Lemma \ref{domination and acl} to see that $\tp(b/C^+k(C^+a))\vdash\tp(b/C^+A)$.

\end{proof}

\section{Appendix: proof of Theorem \ref{HHM 12.15}}
This proof is essentially the same as that given in \cite[Proposition 12.15]{HHM2} in the case of algebraically closed valued fields, and then in \cite[Theorem 2.9]{EHM} for real closed valued fields. In the other two papers, the fields $L$, $M$, and $C$ are assumed to be algebraically (respectively real) closed. We show that this hypothesis is not really needed. We also show that the prior assumption that $C$ is maximal can be replaced with the good separated basis property for $L$ over $C$.  Furthermore, we prove the additional conclusion that if $\sigma$ is the identity on $\RV_L$ as well, then $\sigma$ extends by the identity to all of $\RV_{LM}$.

\begin{reptheorem}{HHM 12.15}
Let $L$, $M$ be subfields of $\Utilde$ with $C\subseteq L\cap M$ a valued subfield, $k(L)$ a regular extension of $k(C)$, and $\Gamma_L/\Gamma_C$ torsion-free. Assume that $\Gamma_L \subseteq \Gamma_M$, that $\{\RV_\gamma(L)\}_{\gamma\in\Gamma_L}$ is algebraically independent from $\{\RV_\gamma(M)\}_{\gamma\in\Gamma_L}$ over $C\Gamma_L$ and that $L$ has the good separated basis property over $C$. Let $\sigma$ be an automorphism of $\Utilde$ mapping $L$ to $L'$ which is the identity on $C$, $\Gamma_L$, and $k_M$.
Then $\sigma|_L$ can be extended to a valued field isomorphism from $LM$ to $L'M$ which is the identity on $M$. Furthermore,  if $\sigma$ is additionally the identity on $\RV_L$, then $\sigma$ may be extended to $LM$ so that it is the identity on $\RV_{LM}$.
\end{reptheorem}

\begin{proof}
In outline, we begin by perturbing the valuation to a finer one, $v'$, which satisfies the 
hypothesis that $\Gamma_{(L,v')}\cap \Gamma_{(M,v')} =\Gamma_{(C,v')}$. We can then apply Proposition~\ref{basicisomorphismtheorem}
to extend $\sigma|_L$ to a $v'$-valued field isomorphism from $LM$ to $L'M$ which extends the identity on $M$. An analysis of the construction shows that this is also a $v$-valued field isomorphism. Finally we use the separated basis hypothesis to show that $\sigma$ is also an isomorphism on $\RV_{LM}$.

The first statement to be proved can be rephrased as saying 
$\tp_{\Ltilde}(L/Ck_M\Gamma_L) \vdash \tp_{\Ltilde}(L/M)$. 
To prove this, we claim that it suffices to prove 
$\tp_{\Ltilde}(L/Ck_M\Gamma_L \Gamma_M) \vdash \tp_{\Ltilde}(L/M)$. 
For, by Lemma~\ref{Remark 3.7 but no e of i}, with $Ck_M$ replacing $C$, and $\Gamma$ replacing $\calS$, we know that 
$\tp_{\Ltilde}(L/Ck_M\Gamma(k_M L)) \vdash \tp_{\Ltilde}(L/Ck_M\Gamma(k_M L) \Gamma(M))$. 
Thus, we just need to verify that $\Gamma(k_M L) = \Gamma(L)=\Gamma_L$. This follows by orthogonality of the value group and residue field. Thus we may assume that $\sigma$ fixes $\Gamma_M$ as well.

Choose $a_1,\ldots, a_r$ from $L$ and $e_1,\ldots,e_r$ from $M$ such that, for each $1\le i\le r$, $v(a_i)=v(e_i)$
and $\{ v(a_i)\}$ forms a $\bQ$-basis for $\Gamma_L$ modulo $\Gamma_C$. Choose $b_1,\ldots,b_s$ from
$L$ such that $\{ \res(b_1),\ldots,\res(b_s)\}$ is a transcendence basis for $k_L$ over $k_C$. By Definition~\ref{RV independence},
the elements
\[
	\res(a_1/e_1),\ldots,\res(a_r/e_r),\res(b_1),\ldots,\res(b_s)
\]
are algebraically independent over $k_M$. For $0\le j \le r$, let
\[
	R^{(j)} = \acl(k_M ,\res(a_1/e_1),\ldots,\res(a_j/e_j),\res(b_1),\ldots,\res(b_s))\cap k_{LM}.
\]
In particular, 
\begin{align*}
	R^{(0)} & = \acl(k_M ,\res(b_1),\ldots,\res(b_s))\cap k_{LM} = \acl(k_M ,k_L )\cap k_{LM} \quad \text{and} \\
	R^{(r)} & = \acl(k_M ,\res(a_1/e_1),\ldots,\res(a_r/e_r), k_L )\cap k_{LM} .
\end{align*}
For each $0\le j\le r-1$, choose a place $p^{(j)}: R^{(j+1)} \to R^{(j)}$ fixing $R^{(j)}$ and such that 
$p^{(j)}(\res(a_{j+1}/e_{j+1})) =0$, which is possible by the algebraic independence of $\res(a_1 /e_1), \dots ,\res(a_r /e_r )$ over 
$k_M$. Also choose a place $p^*:k_{LM} \to R^{(r)}$ fixing $R^{(r)}$. (Later we will show that $k_{LM}=R^{(r)}$ and thus $p^*$ will be seen to be the identity.)
Write $p_v:LM\to k_{LM}$ for the place corresponding to our given valuation $v$. Define
$p_{v'}: LM\to R^{(0)}$ to be the composition
\[
	p_{v'} = p^{(0)}\circ \cdots \circ p^{(r-1)}\circ p^*\circ p_{v} .
\]
Let $v'$ be a valuation associated to the place $p_{v'}$. 
Notice that all the places $p^{(j)}$ and  $p^*$ are the identity on $k_M$, so we may identify 
$(M,v)$ and $(M,v')$, including identifying the value groups $\Gamma_M$ and $\Gamma_{(M,v')}$.
Similarly, the places are all the identity on $k_L$, so the value groups $\Gamma_L$ and $\Gamma_{(L,v')}$
are isomorphic, but we shall see that we cannot simultaneously identify $\Gamma_M$ with $\Gamma_{(M,v')}$ and
$\Gamma_L$ with $\Gamma_{(L,v')}$. 

We now have two valuations $v$ and $v'$ on $LM$. If $x\in M\subseteq LM$, then $v(x)=v'(x)$, and if
$x,y\in L\subseteq LM$ then $v(x) \le v(y)$ implies $v'(x)\le v'(y)$. Furthermore, the construction has
ensured that for any $x\in M$ with $v(x)>0$, and any $w$ such that $\res(w)$ is a nonzero element of $k_{LM}$ mapped to zero by $p^*$,
\[
	0< v'(a_1/e_1) \ll \cdots \ll v'(a_r/e_r) \ll v(w) \ll v'(x),
\]
where $\gamma \ll \delta$ means that $n\gamma < \delta$ for any $n\in\mathbb{N}$ (and hence $\Gamma_{(L,v')}\not=\Gamma_L$). Let $\Delta$ be the 
subgroup of $\Gamma_{(LM,v')}$ generated by $v'(a_1/e_1),\ldots,v'(a_r/e_r)$ together with $v(w)$ for all such $w$. Then $\Delta$
is a convex subgroup of $\Gamma_{(LM,v')}$ and $\Gamma_{(LM,v')} = \Delta\oplus\Gamma_{LM}$, where the right-hand group is ordered
lexicographically. (See, e.g., Theorem 15, Theorem 17, and the associated discussion in Chapter VI of \cite{ZS}). 

To see that $\Gamma_{(L,v')}\cap \Gamma_{(M,v')} = \Gamma_{(C,v')}$, 
let $m\in M$ and $\ell\in L$ be such that $v'(m)=v'(\ell)$.  Set $v'(\frac{a_i}{e_i})=\delta_i$ and $v'(e_i )=\epsilon_i$.  As $(v(a_i))$ generates
$\Gamma_L$ over $\Gamma_C$, and $\Gamma_L$ and $\Gamma_{(L,v')}$ are isomorphic, 
$$v'(\ell)=\sum_{i=1}^{r}p_i v'(a_i ) + \gamma = 
	\sum_{i=1}^{r} p_i \delta_i + \sum_{i=1}^{r} p_i \epsilon_i + \gamma ,
$$ 
where $p_i \in \mathbb{Q}$ and $\gamma \in \Gamma_C$. The set 
$$\{ \delta_1 , \dots ,\delta_r , \epsilon_1 ,\dots ,\epsilon_r \}
$$ 
is algebraically independent over $\Gamma_C$ since $\Gamma_{(LM,v')}=\Delta \oplus \Gamma_{LM}$.  Next, note that 
since $v'(e_i) = v(e_i )$, $\{ v'(e_i ) \}$ forms a $\mathbb{Q}$-basis of $\Gamma_L \subseteq \Gamma_M = \Gamma_{(M,v')}$ 
over $\Gamma_C$.  Let $\mu_1 , \dots ,\mu_t$ be such that $\{ \epsilon_i \} \cup \{ \mu_j \}$ forms a $\mathbb{Q}$-basis of 
$\Gamma_M$ over $\Gamma_C$.  Then 
$$ v'(m)  = \sum_{i=1}^{r}p'_i \epsilon_i + \sum_{i=1}^{t}q_i \mu_i + \gamma',$$ 
where $q_i \in \mathbb{Q}$ and $\gamma' \in \Gamma_C$.  It follows that each $p_i =p'_i =0$ and each $q_i = 0$, 
hence $v'(\ell)=v'(m) \in \Gamma_C$.

Next we must check that $k_{(L,v')}$ and $k_{(M,v')}$ are linearly disjoint. Our definition of $\RV$-independence implies that $k_L$ and $k_M$ are independent over $k_C$, and using that $k_L$ is a regular extension of $k_C$ and Fact~\ref{Lang} we obtain that $k_L$ and $k_M$ are linearly disjoint over $k_C$.  As already observed, the place 
$p^{(0)}\circ \cdots \circ p^{(r-1)}\circ p^*:k_{LM}\to \acl(k_M,k_L)\cap k_{LM}$ is the identity on $k_M$ and 
$k_L$.  Thus this place is also the identity on their compositum, and $k_L k_M=k_{(L,v')}k_{(M,v')}$.  Thus $k_L$ and $k_M$ being linearly disjoint over $k_C$ implies linear disjointness of 
$k_{(L,v')}$ and $k_{(M,v')}$ over $k_{(C,v')}$.

Hence we can apply Corollary~\ref{revisedbasiclemma} to deduce that the isomorphism $\sigma |_{L}$ extends 
to a valued field isomorphism from $({LM},v')$ to $({L'M},v')$ which is the identity on $M$. As $v'$ is a refinement of
$v$, $\sigma$ is also an isomorphism of $({LM},v)$. 

Moreover, by Proposition~\ref{linearlydisjoint}, we know that $\Gamma_{(LM,v')}$ is the sum of $\Gamma_{(L,v')}$ and $\Gamma_{(M,v')}$, and $k_{(LM,v')}=k_{(L,v')}k_{(M,v')}$.  Since $\Gamma_{(LM,v')}$ is also $\Delta\oplus\Gamma_{LM}$, we see both that $\Delta$ must be generated by $\delta_1, \dots, \delta_r$ and that $\Gamma_{LM}=\Gamma_M$.  Since $\Delta$ is generated by $\delta_1, \dots, \delta_r$, in particular this means that there is no $w$ so that $\res(w)$ is a nonzero element mapped to zero by $p^*$.  This implies that $p^*$ is the identity, and that $k_{LM}=\acl(k_M ,\res(a_1/e_1),\ldots,\res(a_r/e_r), k_L )\cap k_{LM}$.

It remains to show that, if $\sigma$ is the identity on $\RV_L$, then it is also the identity on  $\RV_{LM}$. Take an element of $LM$, say $\frac{\Sigma_{i<n}\ell_i m_i}{\Sigma_{j<n}\ell_j m'_i}$.
By the hypothesis, we may assume that the $\{\ell_i\}$ forms a good separated basis over $C$ with respect to $v$ for the subspace it generates, and also with respect to $v'$, since $(L,v)$ and $(L,v')$ are isomorphic. By Lemma~\ref{separatedbasislifts}, this basis is still separated over $M$ with respect to $v'$. Hence, it is even separated over $M$ with respect to $v$, as the following calculation shows:
\[
v(\Sigma_{i<n} m_i\ell_i) = v'(\Sigma_{i<n} m_i\ell_i)/\Delta = (\min_{i<n}\{v'(m_i\ell_i)\})/\Delta = \min_{i<n}\{v'(m_i\ell_i)/\Delta\} = \min_{i<n}\{v(m_i\ell_i)\}.
\]
Since the basis is separated,  we can calculate the $\rv$ of an element of $\RV_{LM}$ as below.
Let $I$ be the set of indices when $v(m_i\ell_i)$ attains its minimum.  Then
\[
\rv(\Sigma_{i<n}m_i\ell_i)=\rv(\Sigma_{i\in I}m_i\ell_i)=\Sigma_{i\in I}\rv(m_i\ell_i)=\Sigma_{i\in I}\rv(m_i)\rv(\ell_i).
\]
Since $\sigma$ fixes $\RV_L$ and $M$, we see that $\sigma$ fixes $\rv$ of any element of the form $\Sigma_{i<n}m_i\ell_i$.  Hence $\sigma$ fixes $\rv$ of any element which is a quotient of such elements, i.e. any element of $LM$.
\end{proof}

\section*{Acknowledgements}

We thank Mariana Vicaria for comments on previous versions of this paper and in particular for pointing out a mistake in the statement of Theorem \ref{dominationovervaluegroup}. We also thank the anonymous referee for very carefully reading the paper and compiling an exhaustive list of comments which resulted in great improvements to the exposition.

\bibliography{Residue_dom}

\end{document}